\documentclass[11pt]{article}

\usepackage[margin=1in]{geometry}

\usepackage{amsmath,amssymb,amsfonts,amsthm}
\usepackage{mathtools}
\usepackage{bm}
\usepackage{nicefrac}

\usepackage{graphicx}
\usepackage{tikz}
\usetikzlibrary{calc,decorations.pathreplacing,arrows.meta,angles,quotes}
\usepackage{subcaption}

\usepackage{booktabs}
\usepackage{multirow}
\usepackage{threeparttable}
\usepackage{adjustbox}

\usepackage{algorithm}
\usepackage{algpseudocode}

\algnewcommand\Input{\item[\textbf{Input:}]}
\algnewcommand\Output{\item[\textbf{Output:}]}

\usepackage{enumitem}
\usepackage{float}
\usepackage{multicol}
\usepackage{pdflscape}
\usepackage{microtype}
\usepackage{xcolor}

\usepackage{natbib}

\usepackage{url}
\usepackage{hyperref}
\theoremstyle{plain}
\newtheorem{proposition}{Proposition}
\newtheorem{lemma}{Lemma}
\newtheorem{corollary}{Corollary}

\theoremstyle{definition}
\newtheorem{definition}{Definition}

\theoremstyle{remark}
\newtheorem{remark}{Remark}

\numberwithin{equation}{subsection}

\title{Change-Point Detection via Piecewise Linear Fitting Using MIP}

\author{
Apoorva Narula, Santanu S. Dey, Yao Xie \\
H. Milton Stewart School of Industrial and Systems Engineering (ISyE),\\
Georgia Institute of Technology, USA\\
\href{mailto:anarula34@gatech.edu}{anarula34@gatech.edu},
\href{mailto:santanu.dey@isye.gatech.edu}{santanu.dey@isye.gatech.edu},
\href{mailto:yao.xie@isye.gatech.edu}{yao.xie@isye.gatech.edu}
}
\date{}  
\begin{document}
\maketitle


\begin{abstract}
We present a new mixed-integer programming (MIP) approach for offline multiple change-point detection by casting the problem as a globally optimal piecewise linear (PWL) fitting problem.  Our main contribution is a family of strengthened MIP formulations whose linear programming (LP) relaxations admit integral projections onto the segment-assignment variables, which encode the segment membership of each data point. This property yields provably tighter relaxations than existing formulations for offline multiple change-point detection. We further extend the framework to multi-dimensional PWL models with shared change-points. Extensive computational experiments on benchmark real-world datasets demonstrate that the proposed formulations achieve reductions in solution times in comparison to the state-of-the-art. 
\end{abstract}


\section{Introduction}
\label{sec:problem_formulation}
Piecewise linear (PWL) fitting involves partitioning ordered observations into contiguous segments, each modeled by a linear function, with the segment boundaries inferred from the data. PWL fitting can be interpreted as a form of change-point detection, where the goal is to identify locations at which either the intercept, the slope, or both parameters abruptly change. Such structural breaks arise in applications across healthcare (\citet{tobiasz2023multivariate}), power systems (\citet{buason2025sample, ahmadi2013piecewise}), manufacturing (\citet{gao2018pvc}), public policy (\citet{wagner2002segmented, dehning2020inferring}), and other domains where detecting regime shifts is essential for interpretation and downstream decision-making. We study the offline setting, in which all observations are available and the goal is to identify change-points retrospectively. 
Mixed-integer programming (MIP) provides an exact framework for this task by explicitly modeling segment assignments and enforcing contiguity. In particular, binary segment-assignment variables determine the segment membership of each observation, enabling direct control over the number and location of change-points.

Despite this modeling flexibility, the scalability of MIP-based approaches is limited by weak linear programming (LP) relaxations, which admit fractional segment assignments and can lead to excessive branch-and-bound effort. 
To address this limitation, this paper introduces strengthened MIP formulations 
for PWL fitting. Our central contribution is a new formulation for partitioning data points into contiguous segments using an extended representation of the segment-assignment variables. We establish that the projection of the linear programming relaxation onto the space of these extended segment-assignment variables is integral, yielding a tighter relaxation. 
Comprehensive computational experiments demonstrate that the proposed formulations consistently achieve global optimality faster than existing benchmarks. 
We further demonstrate that the proposed framework extends naturally to multi-dimensional settings with shared change-points. The main contributions of this work are summarized as follows:
\begin{itemize}
  \item Strengthened mixed-integer programming formulations for offline change-point detection by piecewise linear fitting are developed, whose linear programming relaxations admit an integral projection onto the space of extended segment-assignment variables.
  
  \item A comprehensive computational study is conducted by comparing the proposed formulations with existing MIP benchmarks, demonstrating consistent runtime improvements across multiple loss functions (the $\ell_1$ and $\ell_2$ error norms) and modeling regimes, including both continuous and non-continuous segment formulations and varying numbers of segments.
  
  \item The proposed framework is extended to multi-dimensional offline multiple change-point detection.
  
\end{itemize}
Throughout this paper, we use the term change-point to refer to a structural change in the underlying signal of the data, and the term breakpoint to denote the corresponding decision variable in a PWL model at which model parameters may change.

\subsection{Literature Review}

Advances in computational power and the improved scalability of commercial mixed-integer programming (MIP) solvers have enabled the application of MIP methodologies to increasingly large-scale statistical and machine-learning problems. Motivated by these developments, the present work studies the application of MIP to change-point detection using piecewise linear (PWL) fitting as a unifying framework across the various settings of this statistical problem. 

From an optimization perspective, the key combinatorial decision concerns the optimal partitioning of data points into contiguous segments. 
One of the earliest related formulations is given by \citet{bertsimas2007}, who assign observations to affine components using binary variables. In the univariate setting, this induces a form of piecewise linear fitting; however, contiguity of segment assignments is not enforced, and observations assigned to the same affine component need not be consecutive. 
Subsequent work expanded the scope of piecewise affine fitting to higher-dimensional domains. \citet{toriello2012fitting} develop MIP formulations based on interpolation over a fixed grid, yielding globally optimal solutions relative to a prescribed discretization but requiring the grid to be specified in advance. To avoid explicit triangulation, \citet{cui2018composite} represent the fitted function as the difference of two convex piecewise-affine components, resulting in a non-convex optimization problem.  These approaches either require a predefined fixed grid or sacrifice global optimality.  Another line of research addresses change-point detection through penalization-based methods. In the single-response setting, \citet{prokhorov2025change} control the number of changes using an $\ell_0$ penalty on parameter differences, while total-variation approaches such as  \citet{harchaoui2010multiple} detect mean changes using $\ell_1$ penalization, and \citet{zhang2007modified} use a modified Bayesian information criterion (BIC)  to select the number of mean shifts.  Although computationally efficient, these approaches cannot model continuous piecewise linear fits or be extended to enforce shared segmentation across multiple dimensions of a multivariate time series. 

In the multivariate setting, most existing works, such as \citet{xie2013sequential} and \citet{wang2018high}, primarily focus on mean shifts and do not consider piecewise linear signals. \citet{cao2018multi} consider piecewise linear signals; however, they fix the pre-change slope to be zero. Our model relaxes this assumption. 

In the setting with a one-dimensional domain and co-domain, \citet{rebennack2020}, \citet{kong2020}, and \citet{goldberg2021} develop MIP-based formulations for contiguous segmentation with a fixed number of segments. These encode segmentation decisions using binary variables and model continuity of the fitted function.  A comparative study by \citet{warwicker2022comparison} finds the formulation of \citet{rebennack2020} to be the most computationally effective among existing MIP-based models. Some extensions to these models focus on clustered regression settings (\citet{warwicker2023unified}) and minimizing the number of segments in the continuous piecewise linear fit subject to an approximation error tolerance (\citet{ploussard2024piecewise}).

Building on this line of work, this paper strengthens the segment-assignment formulation for binary-variable models of piecewise linear fitting with a fixed number of contiguous segments. We introduce an extended formulation whose LP relaxation admits an integral projection onto the segment-assignment space, and extend this framework to the multi-dimensional response setting. The remainder of the paper is organized as follows. Section~\ref{sec:formulation_1d} introduces the MIP formulations for one-dimensional piecewise-linear fitting, and extends this to the multi-dimensional change-point detection setting. Section~\ref{sec:Comparison_of_LP_relaxations}  compares the corresponding linear programming relaxations.  Section~\ref{sec:experiments} describes the computational experiments performed to empirically compare the runtime performance of the proposed formulations. Section~\ref{sec:results} presents the results obtained from these experiments, and Section~\ref{sec:conclusion} offers concluding remarks and outlines directions for future research.

\section{Multiple Change-Point Detection for Univariate Sequences}
\label{sec:formulation_1d}

Given a dataset $\{(x_t, y_t)\}_{t=1}^T$ of $T$ observations, where $y_t \in \mathbb{R}$ and $x_t \in \mathbb{R}^d$, we aim to fit a piecewise linear function $f : \mathbb{R}^d \to \mathbb{R}$ with $K$ segments by minimizing a chosen error metric between the fitted and observed responses. Here, $x_t$ denotes the predictor variables. In simple univariate ($d=1$) time-series settings, $x_t$ corresponds to the time index and thus takes non-negative integer values, as in uniformly sampled data; this is the setting considered for most of this work. However, the proposed methodology can be naturally extended to settings in which $x_t$ does not lie on a uniform grid, as in non-uniformly sampled time-series.  We denote by $\mathcal{F}_K$ the family of piecewise linear functions with $K$ segments. We restrict attention to the one-dimensional domain setting ($d=1$), so that $x_t\in\mathbb{R}$ for all $t$. Any function $f\in\mathcal{F}_K$ is parameterized by
$\boldsymbol{\theta}=(\{m_i,c_i\}_{i=1}^K,r_0,\ldots,r_K)$, where
$r_0\leq r_1 \leq \cdots \leq r_K$ are the breakpoints, and satisfies
$f(x\mid\boldsymbol{\theta})=m_i x+c_i$ whenever
$x\in(r_{i-1},r_i]$ for $i=1,\ldots,K$, with the convention that the first interval is $[r_0,r_1]$. We assume that $\boldsymbol{\theta}$ belongs to a parameter space $\boldsymbol{\Theta}$, which imposes lower and upper bounds on the slopes, intercepts, and breakpoint locations based on the observed data. The parameter space $\boldsymbol{\Theta}$ considered in this work is similar to that used by \citet{rebennack2020}, further details on the construction of $\boldsymbol{\Theta}$ are provided in Appendix~\ref{sec:bigM}. Given the function class $\mathcal{F}_K$, the fitting problem is to solve $\min_{f \in \mathcal{F}_K}\sum_{t=1}^T |y_t-f(x_t)|^q$, where $q\in\{1,2\}$ denotes the choice of $\ell_1$ or $\ell_2$ loss. Continuity across segment boundaries may optionally be enforced through the constraints $m_i r_i + c_i = m_{i+1} r_i + c_{i+1}$ for all $i=1,\ldots,K-1$.

By convention, we set $r_0 = x_1$ and $r_K = x_T$. Thus, the representation of a $K$ segment PWL fitting requires $K+1$ breakpoints, of which two (namely, $r_0$ and $r_K$) are fixed. In the MIP-based change-point detection literature (e.g., \citet{goldberg2021, rebennack2020}), continuity at each change-point is typically enforced, yet the broader statistical literature does not generally impose this, as a change may manifest as a jump, a slope change, or both, with continuity warranted in some applications due to physical or engineering constraints but not others.
Motivated by these considerations, both continuous and discontinuous PWL models are studied in the univariate sequence setting. This choice enables a fair comparison between the proposed MIP formulations and existing benchmark methods, which consider this setting. For the multi-dimensional change-point detection setting, continuity is not imposed, consistent with the literature in this area. In the remainder of this work, we refer to the MIP formulations proposed by \citet{goldberg2021} and \citet{rebennack2020} as the Basic and Alternative formulations, respectively. We introduce modular extensions to the segment-assignment components of each of these models, and accordingly refer to our proposed methods as the Extended Basic and Extended Alternative formulations.

\subsection{Basic MIP Formulation}
\label{subsec:Basic}

The formulation of \citet{goldberg2021} relies on binary segment-assignment
variables to encode the partition of the data into linear segments. The complete Basic formulation minimizes the total fitting error under the chosen $\ell_q$ norm, subject to the decision variables and constraints listed below.

\paragraph{Decision variables and parameters}

\begin{itemize}[leftmargin=1.5em]
\item Segment-assignment variables:
For each segment $j\in\{1,\ldots,K\}$ and observation $t\in\{1,\ldots,T\}$, let $\delta_{j,t}\in\{0,1\}$ denote the segment-assignment variable, with $\delta_{j,t}=1$ if observation $t$ belongs to segment $j$ and $\delta_{j,t}=0$ otherwise. These variables define a partition of the data into $K$ contiguous segments.

\item Segment parameters: For each segment $j\in\{1,\ldots,K\}$, let $m_j,c_j\in\mathbb{R}$ denote the slope and intercept of segment $j$.

\item Breakpoint variables: Let $r_j\in\mathbb{R}$, $j=0,\ldots,K$, denote the breakpoint locations, satisfying $r_j\le r_{j+1}$ for all $j$, with $r_0=x_1$ and $r_K=x_T$.

\item Fitted values:
For each observation $t\in\{1,\ldots,T\}$, let $\hat{y}_t\in\mathbb{R}$ denote the fitted value.


\item Big-\(M\) parameters: 
The formulation uses nonnegative constants $M^A_{1,t}$, $M^A_{2,t}$, and $M^A_{3,t}$, for $t=1,\ldots,T$, to deactivate constraints when
$\delta_{j,t} = 0$.
Specifically:
\begin{itemize}[leftmargin=1.5em]
\item $M^A_{1,t}$ relaxes the value-assignment constraints ~\eqref{eq:A2_upper} - ~\eqref{eq:A2_lower}, allowing the
linear function of segment $j$ to differ from $\hat{y}_t$ when observation
$t$ is not assigned to segment $j$.
\item $M^A_{2,t}$ and $M^A_{3,t}$ relax the breakpoint
localization constraints ~\eqref{eq:A3_right} - ~\eqref{eq:A3_left}, ensuring that the bounds imposed
by breakpoints $r_{j-1}$ and $r_j$ are enforced only for the active
segment.
\end{itemize}
The derivation of these constants is discussed in
Appendix~\ref{sec:bigM}.


\end{itemize}



\paragraph{Objective function}
The objective is to minimize the $\ell_q$ fitting error, i.e.,
$\min \sum_{t=1}^T |y_t-\hat{y}_t|^q$, where $q\in\{1,2\}$.

\paragraph{Constraints}

\begin{enumerate}[label=\textbf{(A\arabic*)}, wide]

\item Basic segment-assignment constraints: Each observation must be assigned to exactly one segment.
\begin{equation}
\label{eq:A1_segment_assignment}
    \sum_{j=1}^K \delta_{j,t} = 1,
    \qquad \forall t = 1,\ldots,T.
\end{equation}

\item Value assignment constraints: The fitted value $\hat{y}_t$
must coincide with the linear function associated with segment $j$
whenever $\delta_{j,t} = 1$. Hence, for each segment $j \in \{1,\ldots,K\}$ and observation
$t \in \{1,\ldots,T\}$:
\begin{align}
\label{eq:A2_upper}
    m_j x_t + c_j &\le \hat{y}_t + M^A_{1,t}(1 - \delta_{j,t}),  \\
\label{eq:A2_lower}
    m_j x_t + c_j &\ge \hat{y}_t - M^A_{1,t}(1 - \delta_{j,t}).
\end{align}

\item Breakpoint localization constraints: The breakpoints $r_{j-1}$ and $r_{j}$ must bound all observations assigned to
segment $j$. Hence, for each segment $j \in \{1,\ldots,K\}$ and observation
$t \in \{1,\ldots,T\}$:
\begin{align}
\label{eq:A3_right}
    x_t &\le r_{j} + M^A_{2,t}(1 - \delta_{j,t}), \\
\label{eq:A3_left}
    x_t &\ge r_{j - 1} - M^A_{3,t}(1 - \delta_{j,t}),
\end{align}
with fixed endpoint conditions
\begin{equation}
\label{eq:A3_endpoints}
    r_0 = x_1, 
    \qquad 
    r_{K} = x_T.
\end{equation}
The breakpoints are also required to be non-decreasing in the segment index:
\begin{equation}
\label{eq:A3_breakpoints_ordering}
    r_j \le r_{j+1}, \qquad \forall j = 0,\ldots,K - 1.
\end{equation}
The presence of these breakpoints, which provide upper and lower bounds on all $x_t$ assigned to a segment, together with their nondecreasing order, leads to the assignment of points to segments in a contiguous manner.

\item Basic continuity constraints (optional): To enforce continuity of the fitted function across adjacent segments, the following bilinear constraints are imposed.
\begin{equation}
\label{eq:A4_continuity}
    m_{j} r_j + c_{j}
    =
    m_{j+1} r_{j} + c_{j+1},
    \qquad \forall j = 1,\ldots,K - 1.
\end{equation}

\item Variable domains and parameter bounds: The decision variables and parameters are restricted to the parameter
space $\boldsymbol{\Theta}$, which imposes bounds based on the observed
data.
\begin{align}
    m_j &\in [L^m, U^m] \subset \mathbb{R}, 
    && \forall j \in \{1, \dots, K\}, \notag \\
    c_j &\in [L^c, U^c] \subset \mathbb{R}, 
    && \forall j \in \{1, \dots, K\}, \notag \\
    \hat{y}_t &\in \mathbb{R}, 
    && \forall t \in \{1, \dots, T\}, \notag \\
    \delta_{j,t} &\in \{0,1\}, 
    && \forall j \in \{1, \dots, K\},\; \forall t \in \{1, \dots, T\}.
\end{align}
The construction of $\boldsymbol{\Theta}$ and the associated big-\(M\)
constants $M^A_{1,t}$, $M^A_{2,t}$, and $M^A_{3,t}$ is discussed in
detail in Appendix~\ref{sec:bigM}, which uses the same derivation as \citet{rebennack2020}.

\end{enumerate}

\subsection{Alternative MIP Formulation}
\label{subsec:Alternate}

We next describe the Alternative formulation, which is presented in \citet{rebennack2020}. 
This formulation replaces the bilinear continuity constraints of the Basic model with linear constraints that enforce continuity indirectly, without introducing explicit breakpoint decision variables. 
Continuity is instead modeled by introducing additional variables that capture changes in slope between successive linear segments.
Consequently, two constraint blocks differ from the Basic formulation, namely the segment-assignment constraints and the continuity constraints. We therefore refer to these as the Alternative segment-assignment constraints and the Alternative continuity constraints. 


\paragraph{Additional decision variables and parameters:}
In addition to the decision variables introduced in the Basic formulation, the Alternative formulation uses the following variables and parameters.

\begin{itemize}[leftmargin=1.5em]


\item Slope-direction indicators:
For each segment $j\in\{1,\ldots,K-1\}$, let $\gamma_j\in\{0,1\}$ indicate the direction of slope change, $\gamma_j=1$ if $m_{j+1}\le m_j$, and $\gamma_j=0$ if $m_{j+1}\ge m_j$.


\item Continuity-activation variables:
For each segment $j\in\{1,\ldots,K-1\}$ and observation $t\in\{1,\ldots,T-1\}$, let $\delta^+_{j,t},\delta^-_{j,t}\in[0,1]$ denote activation variables that enforce the appropriate linearized continuity constraints based on the slope direction indicated by $\gamma_j$.


\item Big-\(M\) parameters:
The formulation uses nonnegative constants $M^B_{4,t}$, for $t=1,\ldots,T$, to deactivate the linearized continuity constraints ~\eqref{eq:B2_1} - \eqref{eq:B2_4}. Their derivation is given in Appendix~\ref{sec:bigM} and follows \citet{rebennack2020}.

\end{itemize}

\paragraph{Constraints.}

\begin{enumerate}[label=\textbf{(B\arabic*)}, wide]

\item Alternative segment-assignment constraints:
\begin{align}
\label{eq:B1_assign}
    \sum_{j=1}^K \delta_{j,t} &= 1,
    && \forall t = 1,\ldots,T, \\[4pt]
\label{eq:B1_contiguity}
    \delta_{j+1,t+1} &\le \delta_{j,t} + \delta_{j+1,t},
    && \forall j = 1,\ldots,K-1,\; \forall t = 1,\ldots,T-1, \\[4pt]
\label{eq:B1_first}
    \delta_{1,t+1} &\le \delta_{1,t},
    && \forall t = 1,\ldots,T-1, \\[4pt]
\label{eq:B1_last}
    \delta_{K,t+1} &\ge \delta_{K,t},
    && \forall t = 1,\ldots,T-1.
\end{align}

\item Alternative continuity constraints:  
For all $j = 1,\ldots,K-1$ and $t = 1,\ldots,T-1$, the following inequalities
enforce continuity of the fitted function in a linearized manner:
\begin{align}
\label{eq:B2_1}
    c_{j+1} - c_j 
        &\ge x_t (m_j - m_{j+1}) - M^B_{4,t}(1 - \delta^+_{j,t}), \\
\label{eq:B2_2}
    c_{j+1} - c_j 
        &\le x_{t+1}(m_j - m_{j+1}) + M^B_{4,t+1}(1 - \delta^+_{j,t}), \\
\label{eq:B2_3}
    c_{j+1} - c_j 
        &\le x_t (m_j - m_{j+1}) + M^B_{4,t}(1 - \delta^-_{j,t}), \\
\label{eq:B2_4}
    c_{j+1} - c_j 
        &\ge x_{t+1}(m_j - m_{j+1}) - M^B_{4,t+1}(1 - \delta^-_{j,t}),
\end{align}
where the activation variables are governed by
\begin{align}
\label{eq:B2_5}
    \delta_{j,t} + \delta_{j+1,t+1} + \gamma_j - 2 
        &\le \delta^+_{j,t},  \\
\label{eq:B2_6}
    \delta_{j,t} + \delta_{j+1,t+1} + (1 - \gamma_j) - 2 
        &\le \delta^-_{j,t}.
\end{align}

\end{enumerate}

\begin{remark}[Alternative Method of Contiguous Segment Assignment]\label{rem:AMCSA}
Unlike the Basic formulation that enforces contiguity via explicit breakpoint variables, contiguity is imposed here directly through the assignment variables. For $K = 2$, constraints \eqref{eq:B1_assign}, \eqref{eq:B1_first}, and \eqref{eq:B1_last} ensure that segment assignments are contiguous. For $K > 2$, suppose segment $j > 1$ is assigned non-contiguously, that is, it appears in two disjoint time intervals separated by a gap. Then, by constraint \eqref{eq:B1_contiguity}, segment $j - 1$ must be active immediately before each reappearance of segment $j$, while it cannot be active during periods when segment $j$ is assigned (due to \eqref{eq:B1_assign}). This forces segment $j - 1$ to also be non-contiguous. Repeating this argument propagates non-contiguity backwards to segment $1$. Since segment $1$ is contiguous by \eqref{eq:B1_first}, this yields a contradiction. Hence, all segment assignments must be contiguous.

\end{remark}

\begin{remark}[Alternative Method of Modeling Continuity]\label{rem:AMMC}

Suppose a breakpoint $r$ occurs between observations $x_t$, assigned to segment $j$, and
$x_{t+1}$, assigned to segment $j+1$. Continuity at the breakpoint requires that the two
linear segments intersect at $r$, i.e.,
\[
m_j r + c_j = m_{j+1} r + c_{j+1}.
\]
Solving for $r$ yields
\[
r = \frac{c_{j+1} - c_j}{m_j - m_{j+1}}.
\]
Since the breakpoint lies between $x_t$ and $x_{t+1}$, it must satisfy
\[
x_t \;\le\; \frac{c_{j+1} - c_j}{m_j - m_{j+1}} \;\le\; x_{t+1}.
\]

If $m_j - m_{j+1} \geq 0$, this condition can be equivalently written as
\[
x_t (m_j - m_{j+1}) \;\le\; c_{j+1} - c_j \;\le\; x_{t+1} (m_j - m_{j+1}),
\]
which is modeled by constraints~\eqref{eq:B2_1} - \eqref{eq:B2_2}.  
If instead $m_j - m_{j+1} < 0$, the inequalities reverse, yielding
\[
x_t (m_j - m_{j+1}) \;\ge\; c_{j+1} - c_j \;\ge\; x_{t+1} (m_j - m_{j+1}),
\]
as represented by constraints~\eqref{eq:B2_3} - \eqref{eq:B2_4}. The activation of these constraints is governed by the segment-assignment variables. When a breakpoint occurs between observations $t$ and $t+1$, the assignments satisfy
$\delta_{j,t} = 1$ and $\delta_{j+1,t+1} = 1$. The direction of the slope change between
segments $j$ and $j+1$, encoded by $\gamma_j$, determines which pair of linearized
continuity constraints is enforced. If $m_{j+1} \le m_j$, then $\gamma_j = 1$, and
constraint~\eqref{eq:B2_5} forces $\delta^+_{j,t} = 1$. In this case, the big-$M$ terms in
constraints~\eqref{eq:B2_1} - \eqref{eq:B2_2} vanish, enforcing continuity at the breakpoint.
Conversely, if $m_{j+1} \ge m_j$, then $\gamma_j = 0$, and constraint~\eqref{eq:B2_6} forces
$\delta^-_{j,t} = 1$, activating constraints~\eqref{eq:B2_3} - \eqref{eq:B2_4} instead. Thus, exactly one pair of continuity constraints is activated at each breakpoint,
depending on the relative ordering of the slopes of the adjacent segments.
\end{remark}

\subsection{Proposed Extended Formulation}
\label{sec:formulation_1d_modification}

We now introduce the proposed formulation for enforcing contiguous segment assignment that
replaces the segment-assignment constraint blocks of both the Basic and Alternative
formulations. The key idea is to construct segment membership from a family of
nested, monotonically nonincreasing binary vectors
\[
X_{j, \cdot} = (X_{j,1}, \ldots, X_{j,T}), \qquad j = 1,\ldots,K-1.
\]
For each $j \in \{1,\ldots,K-1\}$ and $t \in \{1,\ldots,T\}$, 
$X_{j,t} \in \{0,1\}$
is a binary indicator that equals 1 if $t$ is assigned to one of the first $j$ segments, and $X_{j,t} = 0$ if it is assigned to segment $j + 1$ or any later segment.
The transition point at
which $X_{j,t}$ switches from $1$ to $0$ therefore identifies the starting index
of segment $j+1$. The collection of vectors $\{X_{j, \cdot}\}_{j=1}^{K-1}$ is required to be nested and
monotonically nonincreasing, ensuring that each segment boundary is
activated at most once and that segment boundaries occur in increasing order. The segment-assignment variables $\delta_{j,t}$ are then derived from differences of the nested vectors, yielding contiguous and ordered segment
blocks. 

\begin{enumerate}[label=\textbf{(C\arabic*)}, wide]

\item Extended segment-assignment constraints:  
The segment-assignment variables $\delta_{j,t}$ are defined in terms of the
nested binaries $X_{j,t}$ as stated below.
\begin{equation}
\label{eq:ext_segment_assignment_1}
\delta_{j,t} =
\begin{cases}
    X_{1,t}, & j = 1, \\[4pt]
    X_{j,t} - X_{j-1,t}, & j = 2,\ldots,K-1, \\[4pt]
    1 - X_{K-1,t}, & j = K,
\end{cases}
\qquad
\forall j = 1,\ldots,K,\;\forall t = 1,\ldots,T.
\end{equation}
To ensure consistent ordering across both segment and time indices, the nested
vectors satisfy
\begin{align}
\label{eq:ext_segment_assignment_2}
    X_{j,t} &\ge X_{j,t+1},
    && \forall j = 1,\ldots,K-1,\;\forall t = 1,\ldots,T-1, \\[4pt]
\label{eq:ext_segment_assignment_3}
    X_{j+1,t} &\ge X_{j,t},
    && \forall j = 1,\ldots,K-2,\;\forall t = 1,\ldots,T.
\end{align}
Together, \eqref{eq:ext_segment_assignment_1} - \eqref{eq:ext_segment_assignment_3}
guarantee that each segment is associated with a single contiguous block of
data points and that segments activate in increasing order from $1$ to $K$.
\end{enumerate}
Replacing the basic segment-assignment block (A1) (constraint~\eqref{eq:A1_segment_assignment}) with the extended block (C1) (constraints~\eqref{eq:ext_segment_assignment_1} - \eqref{eq:ext_segment_assignment_3}) yields the Extended Basic formulation. Similarly, replacing the Alternative segment-assignment block (B1) (constraints~\eqref{eq:B1_assign} - \eqref{eq:B1_last}) with block (C1) yields the Extended Alternative formulation. All remaining constraints of the respective
formulations remain unchanged.

\begin{remark}[Equivalence of Basic and Extended Basic]
\label{rem:EQF}
\label{par:feasible_region_comparison}
The Basic and Extended Basic formulations define identical feasible regions.
As stated earlier, the two models differ only in their treatment of segment-assignment constraints. The Basic formulation uses the assignment constraint (A1) (constraint~\eqref{eq:A1_segment_assignment}), together with the
breakpoint-localization and ordering constraints \eqref{eq:A3_right} - \eqref{eq:A3_breakpoints_ordering}, to enforce segment
contiguity, whereas the Extended Basic formulation replaces (A1) with the
extended segment-assignment constraints (C1), constraints~\eqref{eq:ext_segment_assignment_1} - \eqref{eq:ext_segment_assignment_3}. We first show that every feasible solution to the Extended Basic formulation
is feasible for the Basic formulation.
Summing constraint~\eqref{eq:ext_segment_assignment_1} over the segment index \(j\)
recovers constraint~\eqref{eq:A1_segment_assignment}.
Since all remaining constraints are identical in the two formulations,
any solution feasible for the Extended Basic formulation satisfies all constraints
of the Basic formulation. Conversely, we show that every feasible solution to the Basic formulation is feasible for the Extended Basic formulation.
Let $(\delta, m, c, r, \hat{y})$ be feasible for the Basic formulation. Define
$$X_{j,t} := \sum_{i = 1}^j\delta_{i,t} \ \forall j = 1, \dots, K-1, t = 1, \dots, T.$$
By constraints \eqref{eq:A3_right} - \eqref{eq:A3_breakpoints_ordering}, the assignment $\delta$ induces a contiguous partition of ${1,\dots,T}$ into $K$ segments; using this, it is straightforward to verify that the pair $(\delta, X)$ satisfies constraints~\eqref{eq:ext_segment_assignment_1} - \eqref{eq:ext_segment_assignment_3}.
\end{remark}

\begin{remark}[Alternative vs.\ Extended Alternative]
\label{rem:NEQF}
The Alternative and Extended Alternative formulations differ only in their segment-assignment constraints. The Alternative formulation enforces contiguity through the block (B1) (constraints~\eqref{eq:B1_assign} - \eqref{eq:B1_last}), whereas replacing this block with the extended segment-assignment constraints (C1) yields the Extended Alternative formulation. Every feasible solution to the Alternative formulation is also feasible for the
Extended Alternative formulation since the Alternative formulation admits only contiguous partitions of the data
indices, and every such contiguous partition can be represented using the extended
segment-assignment variables satisfying constraints (C1). Thus, the feasible region of the Alternative formulation is contained within that of the Extended Alternative formulation. However, the reverse inclusion does not hold.
There exist segment assignments that are feasible under the extended segment-assignment constraints but infeasible under the Alternative formulation. To illustrate this, consider the following assignment of the segment variables:
\begin{align}
\delta_{1,\cdot} &= [1,1,0,0,0,0,0,0,0], \notag\\
\delta_{2,\cdot} &= [0,0,1,1,1,0,0,0,0], \notag\\
\delta_{3,\cdot} &= [0,0,0,0,0,0,0,0,0], \notag\\
\delta_{4,\cdot} &= [0,0,0,0,0,1,1,0,0], \notag\\
\delta_{5,\cdot} &= [0,0,0,0,0,0,0,1,1]. \notag
\end{align}
This assignment satisfies the extended segment-assignment constraints
\eqref{eq:ext_segment_assignment_1} - \eqref{eq:ext_segment_assignment_3}.
However, it violates the Alternative contiguity constraint~\eqref{eq:B1_contiguity}. Specifically, setting $(j,t) = (3,5)$ in~\eqref{eq:B1_contiguity} yields
\[
\delta_{4,6} \le \delta_{3,5} + \delta_{4,5} = 0,
\]
whereas $\delta_{4,6} = 1$, resulting in a violation. The violation arises because segment~$3$ is unused, that is, no data points are assigned to it, despite the model allowing up to five segments. Such assignments are permitted by the Extended Alternative formulation, which allows unused segments, but are excluded by the Alternative formulation, whose contiguity constraints implicitly enforce the use of all segments. Consequently, the feasible region of the Extended Alternative formulation strictly contains that of the Alternative formulation.
\end{remark}


\begin{remark}[Same optimal solutions in $m, c, \hat{y}$ space]
\label{rem:SGO}
As discussed above, the Basic and Extended formulations (both Basic and Alternative variants) define equivalent feasible regions, whereas the Alternative formulation's feasible region is a strict subset as it does not permit unused segments. However, consider an optimal solution of any of the Basic, Extended Basic, or Extended Alternative formulations that uses fewer than $K$ segments. Such a solution can be represented within the feasible region of the Alternative formulation by splitting one of the segments into two consecutive segments and assigning them identical slope and intercept values. This is feasible for the Alternative formulation, preserves the objective value, and is therefore an optimal solution for the Alternative formulation. Therefore, all four formulations admit optimal solutions that coincide in the $m, c, \hat{y}$-space.
\end{remark}

\begin{remark}[Unused Segments and $\ell_0$-Regularization]
\label{rem:L0}
Based on Remark~\ref{rem:NEQF}, the Extended Alternative formulation admits feasible piecewise linear fits in which some of the available segments remain unused. By Remark~\ref{rem:EQF}, the Basic and Extended Basic formulations are equivalent, and therefore the Basic formulation also permits such solutions. Consequently, all three formulations, the Basic, Extended Basic, and Extended Alternative formulations can be employed to perform piecewise linear fitting while allowing the number of active segments to be selected by the model through $\ell_0$-type regularization that penalizes the number of segments actually used in the fit, thereby preventing overfitting. In this setting, the parameter $K$ serves as an upper bound on the number of segments, rather than fixing it in advance.
This flexibility facilitates model selection within this mixed-integer programming framework. An explicit MIP formulation for $\ell_0$-regularized piecewise linear fitting is
presented in Appendix~\ref{sec:L0_regularized}.
\end{remark}

\begin{remark}[Multi-dimensional Change-Point Detection]
\label{rem:MultiD}

We adapt the Basic, Alternative, and Extended MIP formulations for
multi-dimensional change-point detection, in which all response components share
a common set of breakpoints and are fit with separate  piecewise-linear functions
over these segments for each dimension. The binary segmentation variables are shared across signals to impose the same segmentation. The slope and intercept variables along with their bounds are applied dimension-wise.  As is
consistent with the multi-dimensional change-point detection literature (\citet{xie2013sequential, wang2018high}), we do not enforce
continuity constraints in this setting. We note that, in the discontinuous PWL setting, the problem can be solved using
dynamic programming (\cite{bai2003computation}). However, since the focus of this paper is on improving the segment-assignment formulation and comparing the resulting MIP formulations, we do not conduct separate experiments using other algorithms. The derivation of the Big-$M$ constants and parameter bounds is performed separately for each dimension, using the same methodology as in the univariate setting as discussed in Appendix~\ref{sec:bigM}.
\begin{equation*}
m_j^d\in[L_d^m,U_d^m],\qquad c_j^d\in[L_d^c,U_d^c],
\qquad j=1,\ldots,K,\;\; d=1,\ldots,D.
\label{eq:multi_bounds}
\end{equation*}
\end{remark}

\section{ Comparison of Linear Programming Relaxations in Delta-Space}
\label{sec:Comparison_of_LP_relaxations}

We study the linear programming relaxations (LPR) of the various change-point detection MIP formulations introduced in Section~\ref{sec:formulation_1d}.  As is standard, we would like to show that the LPR of the proposed formulations is smaller than those for the existing formulations, since this implies possibly stronger lower bounds from the LPR. See discussion in Chapter I.1 in~\cite{nemhauser1988integer}, Chapter 2 in~\cite{conforti2014integer}, and~\cite{shah2026non}. Since all formulations differ in the segment-assignment variables and the associated constraints used, we will study LPRs in the $\delta$-space.

Notation: we distinguish between binary decision variables and their LP-relaxed counterparts using a ``$\sim$''. Specifically, in the LP relaxation, variables $X$ and $\delta$ are denoted as $\widetilde X$ and $\widetilde\delta$, respectively. 

\subsection{Integrality of the Extended (Basic and Alternative) Formulation}
Since both Extended formulations use the same extended variable segment-assignment constraints, it is sufficient to study constraints (C1) (constraints~\eqref{eq:ext_segment_assignment_1} - \eqref{eq:ext_segment_assignment_3}) and their projection onto the $\widetilde\delta$-variable space.  We recall that a polytope is integral if all its vertices are integer-valued, and that the LP relaxation is tightest when it is integral.

\begin{proposition}[Integrality of the (C1) Polyhedron]
\label{prop:integrality-C1-text}
Let \(EB\) denote the polyhedron in the $(\widetilde{X}, \widetilde{\delta})$-space defined by the constraint set (C1), that is,  constraints~\eqref{eq:ext_segment_assignment_1} - \eqref{eq:ext_segment_assignment_3}, together with 
bounds,
\[
\widetilde{X}_{j,t} \in [0,1], \qquad \widetilde{\delta}_{j,t} \in [0,1],
\quad \forall\, j, t.
\]
Then the polyhedron \(EB\) is integral.
\end{proposition}
\begin{proof}

Consider the constraint matrix associated with the constraints in (C1), expressed in the variables \(\{\widetilde{X}_{j,t}, \widetilde{\delta}_{j,t}\}\). After a suitable permutation of columns, this matrix can be written in block form as follows.

\begin{enumerate}
    \item The constraints~\eqref{eq:ext_segment_assignment_1} induce coefficients
    on the \(\widetilde{\delta}_{j,t}\) variables that form an identity submatrix.

    \item The constraints~\eqref{eq:ext_segment_assignment_1} - 
    \eqref{eq:ext_segment_assignment_3} induce coefficients on the
    \(\widetilde{X}_{j,t}\) variables such that each row contains at most two nonzero
    entries, each equal to \(+1\) or \(-1\), and whenever two nonzero entries
    appear in a row, they have opposite signs.
    Constraint matrices with this structure are well known to be totally
    unimodular.
\end{enumerate}

Total unimodularity is preserved under column permutations and under horizontal concatenation with an identity matrix. Consequently, the full constraint matrix associated with (C1) is totally unimodular. 
The bound constraints
$$ 0 \le \widetilde{X}_{j,t} \le 1, \qquad 0 \le \widetilde{\delta}_{j,t} \le 1,$$
can be expressed as linear inequalities with identity coefficient matrices (up to sign) in the corresponding variable columns.
Again, appending rows corresponding to identity matrices (or their negations) 
preserves total unimodularity. Therefore, the full constraint matrix associated with (C1), together with the variable bound constraints, is totally unimodular. Since the right-hand-side vector is integral, it follows that $EB$ is an integral polyhedron.
\end{proof}

Since the projection of an integral polytope onto a subset of its variables is also an integral polytope, the following corollary holds.

\begin{corollary}\label{cor:EBI}
    Let \(EB\) denote the polyhedron in the $(\widetilde{X}, \widetilde{\delta})$-space defined by the constraint set (C1), that is, constraints~\eqref{eq:ext_segment_assignment_1} - \eqref{eq:ext_segment_assignment_3}, together with 
bounds,
\[
\widetilde{X}_{j,t} \in [0,1], \qquad \widetilde{\delta}_{j,t} \in [0,1],
\quad \forall\, j, t.
\]
Then the projection of \(EB\) onto the space of $\widetilde{\delta}$ variables is integral.
\end{corollary}
\subsection{Fractionality of Basic and Alternative formulations}
In this section, we establish that the Basic and the Alternative formulations are not integral. We first consider the Basic formulation. Since the Basic and Extended Basic formulations have the same integer feasible region, in view of Corollary~\ref{cor:EBI}, it suffices to show that the LP relaxation of the Extended Basic formulation is strictly contained in that of the Basic formulation. In order to prove this result, we require a few preliminary definitions and properties.

\begin{definition}[Total Variation of a Vector]
\label{def:tv-vector}
Let $x \in \mathbb{R}^T$ be a vector indexed by $t=1,\ldots,T$. The total variation of $x$ is defined as $\mathrm{TV}(x) := \sum_{t=2}^T |x_t - x_{t-1}|.$
\end{definition}

\begin{lemma}[Total Variation Bound for Differences] \label{lem:tv-difference-vector}
Let $x,y \in \mathbb{R}^T$, and define $z = x - y$. Then $\mathrm{TV}(z) \le \mathrm{TV}(x) + \mathrm{TV}(y).$
\end{lemma}
\begin{proof}

By definition,
\begin{eqnarray*}
    \mathrm{TV}(z) &=& \sum_{t=2}^T |z_t - z_{t-1}| = \sum_{t=2}^T |(x_t - x_{t-1}) - (y_t - y_{t-1})| \\
    &\leq&\sum_{t=2}^T |x_t - x_{t-1}| + \sum_{t=2}^T |y_t - y_{t-1}| =\mathrm{TV}(x) + \mathrm{TV}(y),
\end{eqnarray*}
where the inequality is obtained by using the triangle inequality term-wise.
\end{proof}

\begin{proposition}[Strict Containment of the Extended Basic Projection]\label{prop:projection_subset_Basic}
Let $\widetilde{B}$ and $\widetilde{EB}$ denote the polytopes obtained by projecting the LP relaxations of the Basic and Extended Basic formulations, respectively, onto the $\delta$-variable space. 
Then, $ \widetilde{EB} \subsetneq \widetilde{B} $. As a consequence, the Basic formulation is not integral in the $\delta$-space.
\end{proposition}

\begin{proof}

By Corollary~\ref{cor:EBI} and Remark~\ref{rem:EQF}, we have that $\widetilde{EB} \subseteq \widetilde{B}$. It remains to show that this inclusion is strict. We demonstrate that there exist
$\widetilde{\delta}$-patterns that satisfy the LPR of the Basic formulation but cannot
arise from any feasible solution of constraint set (C1) (constraints~\eqref{eq:ext_segment_assignment_1} - \eqref{eq:ext_segment_assignment_3}); hence, they do not belong to $\widetilde{EB}$. The key observation is that the relaxed variables $\widetilde{X}_{j,t}$ appearing in (C1) form nested, monotonically non-increasing relaxed sequences as a consequence of constraints
\eqref{eq:ext_segment_assignment_2} - \eqref{eq:ext_segment_assignment_3}.
This imposes a restriction on the total variation of differences
between such sequences.  In particular, observe that
\begin{eqnarray*}
\mathrm{TV}(\widetilde{X}_{j,\cdot})
= \sum_{t=2}^T \left| \widetilde{X}_{j,t} - \widetilde{X}_{j,t-1} \right|
= \sum_{t=2}^T \left( \widetilde{X}_{j,t-1} - \widetilde{X}_{j,t} \right)
= \widetilde{X}_{j,1} - \widetilde{X}_{j,T}
\leq 1,
\end{eqnarray*}
where the first equality follows from the fact that $\widetilde{X}_{j,t -1} \geq \widetilde{X}_{j,t},$ and the last inequality follows from the fact that $\widetilde{X}_{j,t} \in [0,1]$. We now apply Lemma~\ref{lem:tv-difference-vector} to obtain
$$\mathrm{TV}(\widetilde{\delta}_{j,\cdot}) \leq \mathrm{TV}(\widetilde{X}_{j,\cdot}) + \mathrm{TV}(\widetilde{X}_{j-1,\cdot}) \leq 2,$$
where we use the fact that $\widetilde{\delta}_{j,t} = \widetilde{X}_{j,t} - \widetilde{X}_{j-1,t}$.
Consequently, any $\widetilde{\delta}$-sequence whose total variation exceeds~2 cannot belong to $\widetilde{EB}$. We provide an explicit example in which the sequences $\widetilde{\delta}_{2,\cdot}$ and $\widetilde{\delta}_{3,\cdot}$ attain total variation exceeding~2, rendering them infeasible for the LPR of the Extended Basic formulation while remaining feasible for the Basic relaxation. The
corresponding segment-assignment vectors are:
\[
\widetilde{\delta}_{1,\cdot} = [0.5, 0.5, 0.5, 0.5, 0.5, 0.5, 0.5, 0.5],
\]
\[
\widetilde{\delta}_{2,\cdot} = [0.4, 0.1, 0.4, 0.1, 0.4, 0.1, 0.4, 0.1],
\]
\[
\widetilde{\delta}_{3,\cdot} = [0.1, 0.4, 0.1, 0.4, 0.1, 0.4, 0.1, 0.4],
\]
\[
\widetilde{\delta}_{4,\cdot} = [0,0,0,0,0,0,0,0].
\]

As shown in Appendix~\ref{sec:prop_continued}, this $\widetilde{\delta}$ provides a feasible solution satisfying all the constraints of the LPR of the Basic formulation; hence, the above segment-assignment vectors belong to $\widetilde{B}$. 
\end{proof}

Next, we consider the Alternative formulation. Recall that the feasible region of the Extended formulation differs from that of the Alternative formulation, so no direct comparison between the two is made here. Rather, we aim to establish that the segment-assignment constraints of the Alternative formulation do not constitute an integral polytope. 

\begin{proposition}[Non-Integrality of polytope defined by (B1) Constraints]
\label{prop:Nonintegrality-B1-text}
Let \(A\) denote the polyhedron defined by the constraint set (B1),
that is, constraints~\eqref{eq:B1_assign} - \eqref{eq:B1_last},
together with bounds:
\[
\widetilde{\delta}_{j,t} \in [0,1],
\quad \forall\, j = 1,\ldots,K,\; t = 1,\ldots,T.
\]
Then the polyhedron \(A\) is not integral.
\end{proposition}

\begin{proof}

We consider an instance with $T = 6$ data points and $K = 4$ candidate segments and show that this instance admits a fractional extreme point. The segment–assignment variables take the following values:
\begin{eqnarray*}
\widetilde{\delta}_{1,\cdot} &=& [0.5,\; 0.5,\; 0.5,\; 0,\; 0,\; 0],\\
\widetilde{\delta}_{2,\cdot} &=& [0,\; 0.5,\; 0,\; 0.5,\; 0.5,\; 0.5],\\
\widetilde{\delta}_{3,\cdot} &=& [0.5,\; 0,\; 0.5,\; 0,\; 0,\; 0], \\
\widetilde{\delta}_{4,\cdot} &=& [0,\; 0,\; 0,\; 0.5,\; 0.5,\; 0.5].
\end{eqnarray*}

It is straightforward to verify that this point is feasible for the LPR of the Alternative formulation. To establish that this point is an extreme point of $A$, we identify $24$ linearly independent constraints that are tight as the ambient space is $\mathbb{R}^{24}$ (corresponding to the variables $\{\widetilde{\delta}_{j,t}\}_{j=1,\ldots,4;\, t=1,\ldots,6}$).

\begin{enumerate}
\item {Assignment constraints.}  
There are $6$ equality constraints of type~\eqref{eq:B1_assign}:
\[
\sum_{j=1}^4 \widetilde{\delta}_{j,t} = 1, \qquad \forall t = 1,\ldots,6.
\]

\item {Contiguity constraints.}  
The following $5$ constraints of type~\eqref{eq:B1_contiguity} are tight:
\[
\begin{aligned}
\widetilde{\delta}_{2,2} &= \widetilde{\delta}_{1,1} + \widetilde{\delta}_{2,1}, \\
\widetilde{\delta}_{2,4} &= \widetilde{\delta}_{1,3} + \widetilde{\delta}_{2,3}, \\
\widetilde{\delta}_{2,5} &= \widetilde{\delta}_{1,4} + \widetilde{\delta}_{2,4}, \\
\widetilde{\delta}_{2,6} &= \widetilde{\delta}_{1,5} + \widetilde{\delta}_{2,5}, \\
\widetilde{\delta}_{3,3} &= \widetilde{\delta}_{2,2} + \widetilde{\delta}_{3,2}, 
\end{aligned}
\]

\item {Boundary constraints.}  
The following constraint of type~\eqref{eq:B1_first} - \eqref{eq:B1_last} is tight: $\widetilde{\delta}_{1,1} = \widetilde{\delta}_{1,2}$.

\item {Variable bound constraints.}  
Finally, the following $12$ variable bounds are tight:
$$\begin{aligned}
\widetilde\delta_{1,4}&=0, & \widetilde\delta_{1,5}&=0, & \widetilde\delta_{1,6}&=0,\\
\widetilde\delta_{2,1}&=0, & \widetilde\delta_{2,3}&=0,\\
\widetilde\delta_{3,2}&=0, & \widetilde\delta_{3,4}&=0, & \widetilde\delta_{3,5}&=0, & \widetilde\delta_{3,6}&=0,\\
\widetilde\delta_{4,1}&=0, & \widetilde\delta_{4,2}&=0, & \widetilde\delta_{4,3}&=0.
\end{aligned}$$
\end{enumerate}

It remains to justify that the above $24$ tight constraints are linearly independent. We do this by showing that they uniquely determine all $24$ variables.

First, the 12 tight bound constraints immediately fix
$\widetilde\delta_{1,4} =\widetilde \delta_{1,5}= \widetilde \delta_{1,6}=\widetilde\delta_{2,1}=\widetilde\delta_{2,3} =\widetilde\delta_{3,2}=\widetilde\delta_{3,4}=\widetilde\delta_{3,5}=\widetilde\delta_{3,6}= \widetilde\delta_{4,1}=\widetilde\delta_{4,2}= \widetilde\delta_{4,3}=0.$ The boundary constraints imply
$\widetilde \delta_{1,1}=\widetilde\delta_{1,2}$.
Using the tight contiguity constraints (and the tight bound constraints), we now obtain $\widetilde\delta_{2,2}=\widetilde\delta_{1,1}$,
$\widetilde\delta_{2,4}=\widetilde\delta_{1,3}$, $\widetilde\delta_{2,5}=\widetilde\delta_{2,4}$,
$\widetilde\delta_{2,6}=\widetilde\delta_{2,5}$,
$\widetilde\delta_{3,3}=\widetilde\delta_{2,2}$. 

The assignment constraints then determine the remaining variables:
$$\widetilde\delta_{1,1}+\widetilde\delta_{3,1}=1,$$
$$\widetilde\delta_{1,2}+\widetilde\delta_{2,2}=1, $$
$$\widetilde\delta_{1,3}+\widetilde\delta_{3,3}=1,$$
$$\widetilde\delta_{2,4}+\widetilde\delta_{4,4}=1,$$
$$\widetilde\delta_{2,5}+\widetilde\delta_{4,5}=1,$$
$$\widetilde\delta_{2,6}+\widetilde\delta_{4,6}=1.$$
Together these equations uniquely determine all variables, yielding exactly
the point $\widetilde{\delta}$. 
\end{proof}
We summarize the key insights from the analysis in this section as follows. The projection of the Extended formulations onto the segment-assignment variables is integral, while the same is not true for the Basic and the Alternative formulations. While these analyses provide insight into potential differences in LP relaxation strength and solver convergence behavior, the full MIP formulations also differ in how slopes and intercepts are assigned to each segment in the piecewise linear fitting problem. As a result, the polyhedral comparisons in this section are not necessarily conclusive with respect to overall computational performance (also see~\cite{shah2026non}), which is therefore assessed empirically in Section~\ref{sec:experiments}.

\section{Experiments}
\label{sec:experiments}
We conduct computational experiments to compare the relative runtime performance
of the proposed extended segment-assignment constraint formulation against the
existing segment-assignment formulations, namely the Basic and Alternative
formulations. These experiments complement and further substantiate the
theoretical analysis provided in Section~\ref{sec:Comparison_of_LP_relaxations}. All experiments were conducted using Gurobi Optimizer 13.0.2 on an Apple M3 processor with eight CPU cores. Each optimization run was restricted to a single Gurobi thread to control for variability arising from parallel execution, processor resource contention, and processor frequency scaling, thereby enabling fair and reproducible comparisons across formulations. All results were replicated using three random seeds (0, 10000, and 20000). The reported runtimes and relative optimality gaps correspond to the arithmetic mean across these three independent runs. A wall-clock time limit of 2000 seconds was imposed for every run, and non-default parameters were kept minimal to avoid solver-specific tuning effects. The source code used to generate all computational results presented in this paper, together with the corresponding experimental outputs, is publicly available at \url{https://github.com/ApoorvaGOR/MIP_for_Change_point_Detection}.

\subsection{Experiments: Univariate sequence without continuity}
\label{subsec:computational_1D_without_continuity}

The comparison under the non-continuous setting focuses solely on the computational performance of the segment-assignment constraint blocks across the MIP formulations. Accordingly, in this setting, we compare three MIP formulations corresponding to the Basic, Alternative, and Extended segment-assignment constraint blocks. For the time-series data, we use four stock tickers, namely Apple Inc. (AAPL), Amazon.com Inc. (AMZN), Alphabet Inc. (GOOGL), and Microsoft Corporation (MSFT), with daily closing prices (starting January~1, 2015) obtained from the \texttt{yfinance} Python library. We consider time-series lengths $T \in \{100, 200, 300, 400, 500\}$, segment counts $K \in \{2,3,4,5,6\}$, and objective functions under both $\ell_1$- and $\ell_2$-error norms. This yields 200 distinct experimental settings.


\subsection{Experiments: Univariate sequence with continuity}
\label{subsec:computational_1D_with_continuity}

Since the Extended segment-assignment block is modular, it can be paired with either the Basic or Alternative continuity-enforcing constraints, yielding four MIP formulations: Basic, Alternative, Extended Basic, and Extended Alternative. We compare these across time-series lengths $T \in \{100, 120, 140, 160, 180\}$ and segment counts $K \in \{2,3,4,5\}$ under both $\ell_1$- and $\ell_2$-error norms, using daily closing prices (starting January~1, 2015) for stock tickers Apple Inc.\ (AAPL), Amazon.com Inc. (AMZN), Alphabet Inc. (GOOGL), Johnson \& Johnson (JNJ), and Microsoft Corporation (MSFT). Here, shorter time horizons are considered because, for longer time-series lengths and higher segment counts, the formulations more frequently hit the runtime limit. This yields 200 distinct experimental settings.

\subsection{Experiments: Multi-dimensional change-point detection}
\label{sec:multiD_experiments}

We evaluate the multi-dimensional formulations using daily closing prices (starting January~1,~2015) for ten stock tickers, namely Apple Inc.\ (AAPL), Microsoft Corporation (MSFT), Amazon.com Inc.\ (AMZN), Alphabet Inc.\ (GOOGL), Meta Platforms, Inc.\ (META), Tesla, Inc.\ (TSLA), NVIDIA Corporation (NVDA), JPMorgan Chase \& Co.\ (JPM), Berkshire Hathaway Inc.\ Class B (BRK-B), and Johnson \& Johnson (JNJ). A $D$-dimensional signal is constructed by jointly fitting the first $D$ tickers from this list. We consider $D \in \{2,4,6,8,10\}$, $T \in \{100,200,300,400,500\}$, and $K \in \{2,3,4,5\}$. Since this setting also considers piecewise linear fitting without continuity, as in Section~\ref{subsec:computational_1D_without_continuity}, we again compare the Basic, Alternative, and Extended formulations, yielding a total of 200 experimental configurations.

\section{Results}
\label{sec:results}
\begin{figure}[H]
\centering

\begin{subfigure}[b]{0.31\textwidth}
    \centering
    \includegraphics[width=\textwidth]{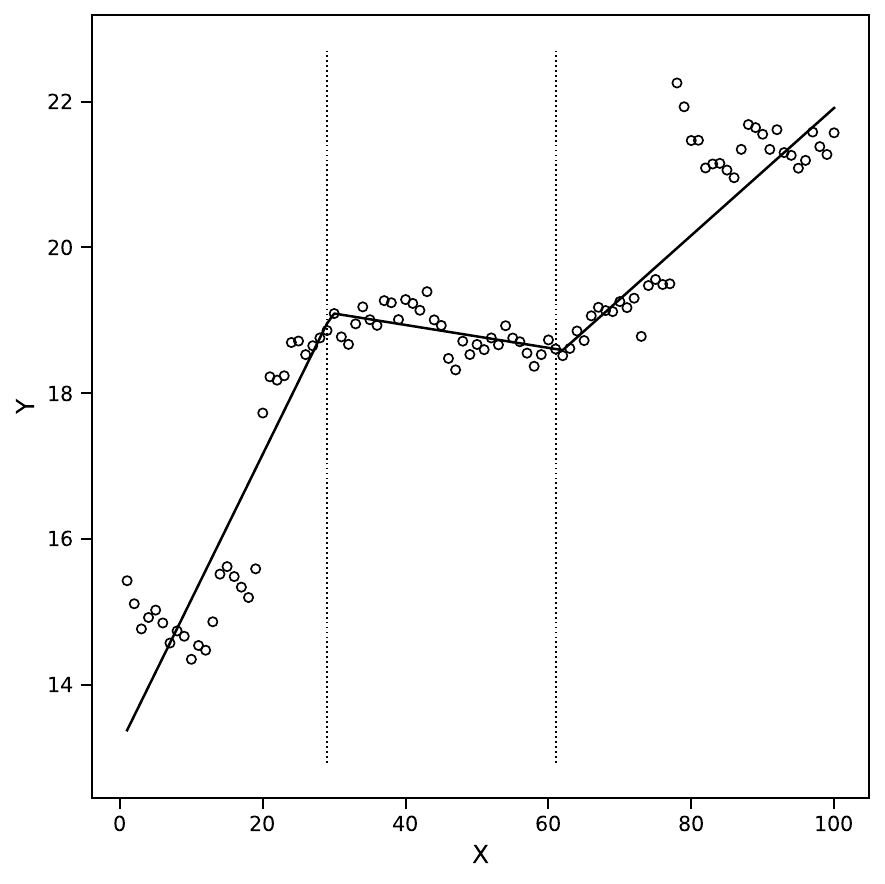}
    \caption{$T = 100$, $K = 3$}
    \label{fig:Univariate_with_continuity}
\end{subfigure}
\hfill
\begin{subfigure}[b]{0.31\textwidth}
    \centering
    \includegraphics[width=\textwidth]{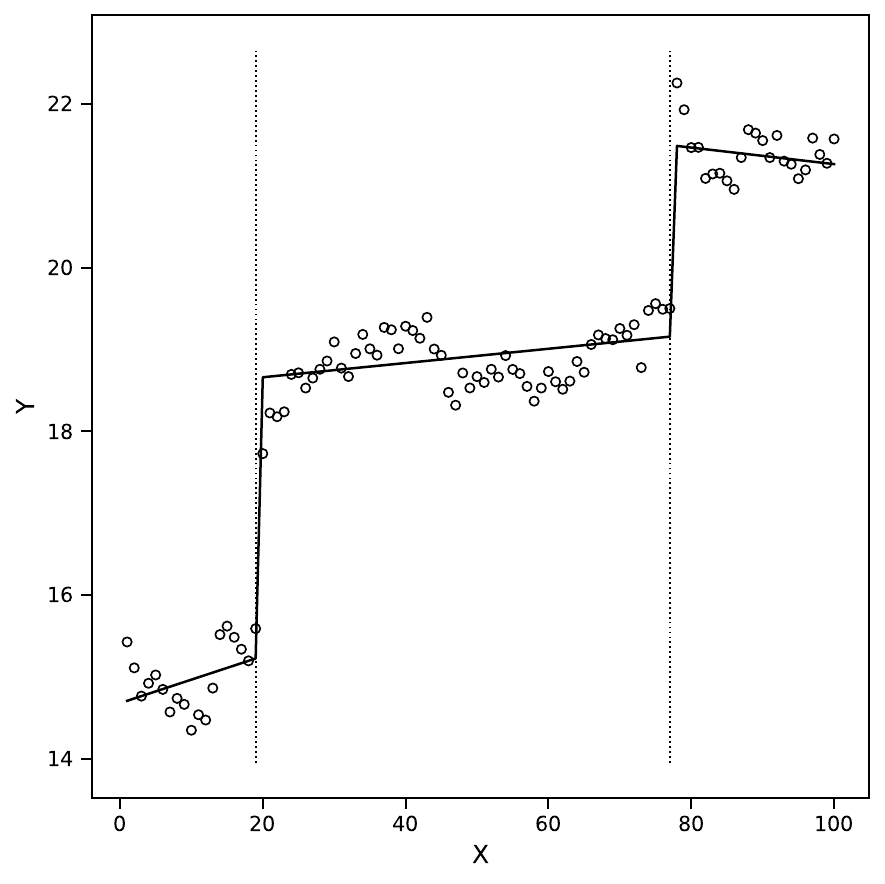}
    \caption{$T = 100$, $K = 3$}
    \label{fig:Univariate_without_continuity}
\end{subfigure}
\hfill
\begin{subfigure}[b]{0.31\textwidth}
    \centering
    \includegraphics[width=\textwidth]{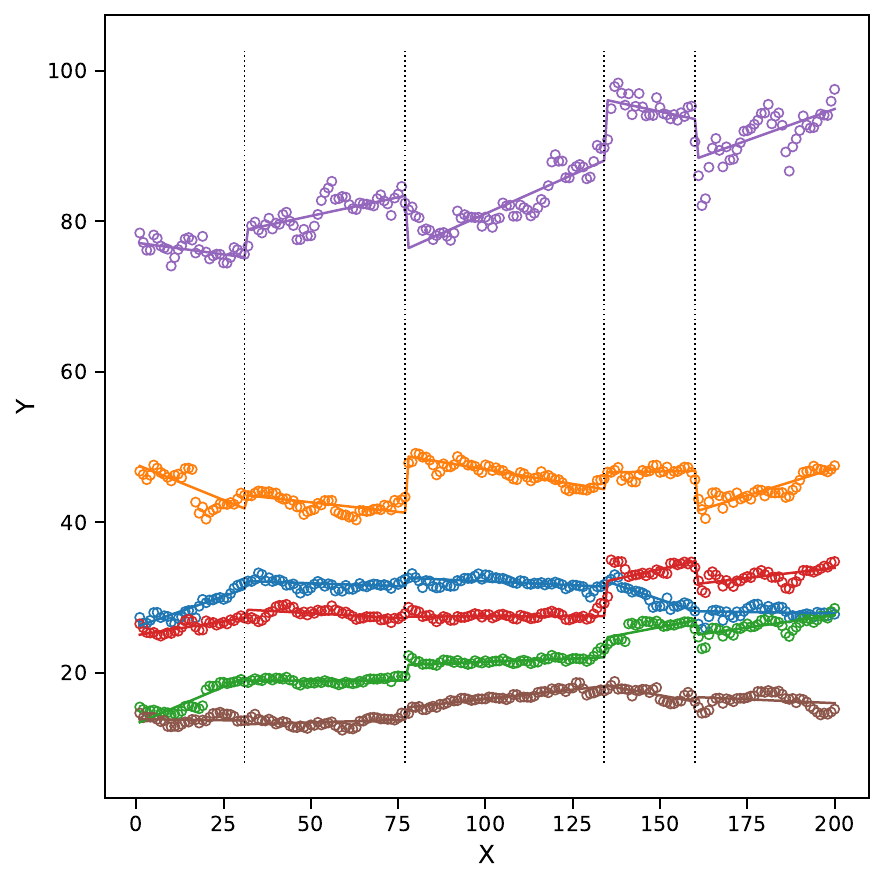}
    \caption{$T = 200$, $K = 5$, $D = 6$}
    \label{fig:Multi_CPD}
\end{subfigure}

\caption{
Illustration of piecewise-linear fits obtained through solving the MIP formulations while minimizing the $\ell_1$ error norm: 
\textnormal{(a)} univariate fitting with continuity based on Amazon (AMZN) closing prices, 
\textnormal{(b)} univariate fitting without continuity using the same data as (a), and 
\textnormal{(c)} multi-dimensional fitting with shared change-points across \(D=6\) time series.
}

\label{fig:PWL_CPD_examples}
\end{figure}

\begin{table}[H]
\centering
\small
\renewcommand{\arraystretch}{1.2}
\setlength{\tabcolsep}{6pt}

\caption{
Summary of runtime performance across all computational experiments. Each entry reports the number of instances in which a formulation attained the lowest arithmetic mean runtime across three random seeds. If all formulations reached the 2000-second time limit for an instance, the comparison is instead based on the arithmetic mean relative optimality gap.
}
\label{tab:unified_benchmark_summary}

\begin{tabular}{l l c c c c c}
\toprule
\textbf{Experiment Class} & \textbf{Setting} & \textbf{Norm}
& \textbf{Basic} & \textbf{Alternative} & \textbf{Extended} & \textbf{Total} \\
\midrule

\multirow{4}{*}{\parbox{3.1cm}{\centering Single-Dimensional \\ PWL Fitting}}

& \multirow{2}{*}{Without Continuity}
  & $\ell_1$ & 4  & 36 & \textbf{60} & 100 \\
& & $\ell_2$ &  0 & 23 & \textbf{77} & 100 \\

\cmidrule(lr){2-7}

& \multirow{2}{*}{With Continuity}
  & $\ell_1$ & 42 & 8 & \textbf{50} & 100 \\
& & $\ell_2$ & 0  & 28 & \textbf{72} & 100 \\

\midrule

\multirow{2}{*}{\parbox{3.1cm}{\centering Multi-Dimensional \\ PWL Fitting}}
& \multirow{2}{*}{}
  & $\ell_1$ & 5  & 28 & \textbf{67} & 100 \\
& & $\ell_2$ & 11 & 40 & \textbf{49} & 100 \\

\midrule

\multicolumn{3}{l}{\textbf{Overall Total (All Settings)}}
& 62 & 163 & \textbf{375} & \textbf{600} \\
\bottomrule
\end{tabular}

\end{table}

Tables~\ref{tab:unified_benchmark_summary} and~\ref{tab:unified_total_runtime} provide a unified view of solver performance across all experiment classes and settings. As summarized in Table~\ref{tab:unified_benchmark_summary}, the Extended formulations achieve the fastest convergence in the majority of instances overall (375 out of 600), outperforming both Basic (62/600) and Alternative (163/600). For the single-dimensional PWL fitting experiments with continuity, the Extended column aggregates results from two formulations, namely Extended Basic and Extended Alternative. Under the $\ell_1$ norm, Extended Basic is the fastest formulation in 30 instances, while Extended Alternative is fastest in the remaining 20 instances. Under the $\ell_2$ norm, Extended Basic is fastest in 10 instances, whereas Extended Alternative is fastest in 62 instances. Complementing this instance-wise summary, Table~\ref{tab:unified_total_runtime} shows that the Extended formulations also achieve the lowest cumulative runtime across all settings (171884.21~s, versus 216379.52~s for Alternative and 246307.74~s for Basic). For the single-dimensional PWL fitting experiments with continuity, the Extended column reports the smaller cumulative runtime of the Extended Basic and Extended Alternative formulations for each norm. Specifically, under the $\ell_1$ norm, Extended Basic requires a cumulative runtime of 28709.02~s compared with 41633.72~s for the Basic formulation, while Extended Alternative requires 53708.13~s compared with 62613.81~s for the Alternative formulation. Similarly, under the $\ell_2$ norm, Extended Basic requires 40959.35~s compared with 68502.35~s for the Basic formulation, whereas Extended Alternative requires 22583.86~s compared with 23832.66~s for the Alternative formulation. Thus, both Extended Basic and Extended Alternative consistently outperform their respective benchmark formulations in terms of cumulative runtime for both $\ell_1$ and $\ell_2$ norms. These aggregate summaries are obtained by consolidating the detailed per-instance runtimes and relative optimality gaps reported in Appendix Tables~\ref{tab:runtime_1D_without_continuity} (univariate without continuity), \ref{tab:runtime_1D_with_continuity} (univariate with continuity), and \ref{tab:runtime_MultiD_without_continuity} (multi-dimensional shared change-points). Figure~\ref{fig:PWL_CPD_examples} visually illustrates representative fits from each regime, namely univariate with and without continuity (Figures~\ref{fig:Univariate_with_continuity} - \ref{fig:Univariate_without_continuity}) and multi-dimensional fitting with shared change-points (Figure~\ref{fig:Multi_CPD}).

\begin{table}[H]
\centering
\small
\renewcommand{\arraystretch}{1.25}
\setlength{\tabcolsep}{7pt}

\caption{
Unified comparison of total runtime (in seconds)  across all computational experiments and random seeds. Instances that reached the runtime limit are counted as having runtime 2000 seconds.
}
\label{tab:unified_total_runtime}

\begin{tabular}{l l c c c c}
\toprule
\textbf{Experiment Class} & \textbf{Setting} & \textbf{Norm}
& \textbf{Basic} & \textbf{Alternative} & \textbf{Extended} \\
\midrule

\multirow{4}{*}{\parbox{3.2cm}{\centering Single-Dimensional \\ PWL Fitting}}

& \multirow{2}{*}{Without Continuity}
  & $\ell_1$ & 31065.49 & 30500.25 & \textbf{28767.86}  \\
& & $\ell_2$ & 11074.89 & 7910.59 & \textbf{4878.22} \\

\cmidrule(lr){2-6}

& \multirow{2}{*}{With Continuity}
  & $\ell_1$ & 41633.72 & 62613.81 & \textbf{28709.02} \\
& & $\ell_2$ & 68502.35 & 23832.66 & \textbf{22583.86} \\

\midrule

\multirow{2}{*}{\parbox{3.2cm}{\centering Multi-Dimensional \\ PWL Fitting}}
& \multirow{2}{*}{}
  & $\ell_1$ & 58739.26 & 58630.37 & \textbf{51875.88} \\
& & $\ell_2$ & 35292.03 & \textbf{32891.84} & 35069.37 \\

\midrule

\multicolumn{3}{l}{\textbf{Overall Total (All Settings)}}
& 246307.74 & 216379.52 & \textbf{171884.21} \\
\bottomrule
\end{tabular}

\end{table}

\section{Conclusion}
\label{sec:conclusion}
This work addresses the statistical problem of change-point detection through a unified MIP-based framework for piecewise-linear fitting. The key contribution is a novel set of constraints for partitioning the data
into contiguous segments, where each segment is fit by a linear function. These
constraints exhibit favorable theoretical properties that contribute to faster
solver runtimes. In particular, relaxing the integrality requirements on the
segment-assignment variables yields LP relaxations whose feasible region forms
an integral polytope. We provide theoretical arguments to explain why the proposed segment-assignment constraints lead to improved computational performance and complement these results with comprehensive computational experiments that empirically verify the claimed improvements. Promising avenues for future research include enhancing scalability through decomposition techniques and developing online variants of the proposed change-point detection framework.

\section*{Acknowledgments} The work of Narula and Xie is partially supported by National Science Foundation DMS-2220495.

\bibliographystyle{plainnat}
\bibliography{refs}

@article{shah2026non,
  title={Non-monotonicity of branching rules with respect to linear relaxations},
  author={Shah, Prachi and Dey, Santanu S and Molinaro, Marco},
  journal={INFORMS Journal on Computing},
  volume={38},
  number={1},
  pages={53--66},
  year={2026},
  publisher={INFORMS}
}

@book{nemhauser1988integer,
  title     = {Integer and Combinatorial Optimization},
  author    = {Nemhauser, George L. and Wolsey, Laurence A.},
  year      = {1988},
  publisher = {John Wiley \& Sons},
  address   = {New York},
  series    = {Wiley-Interscience Series in Discrete Mathematics and Optimization},
  isbn      = {978-0-471-82819-8},
  pages     = {763}
}

@book{conforti2014integer,
  title     = {Integer Programming},
  author    = {Conforti, Michele and Cornu{\'e}jols, G{\'e}rard and Zambelli, Giacomo},
  series    = {Graduate Texts in Mathematics},
  volume    = {271},
  year      = {2014},
  publisher = {Springer},
  address   = {Cham},
  doi       = {10.1007/978-3-319-11008-0},
  isbn      = {978-3-319-11007-3}
}

@article{goldberg2021,
  title={MINLP formulations for continuous piecewise linear function fitting},
  author={Goldberg, Noam and Rebennack, Steffen and Kim, Youngdae and Krasko, Vitaliy and Leyffer, Sven},
  journal={Computational Optimization and Applications},
  volume={79},
  number={1},
  pages={223--233},
  year={2021},
  publisher={Springer}
}

@article{rebennack2020,
  title={Piecewise linear function fitting via mixed-integer linear programming},
  author={Rebennack, Steffen and Krasko, Vitaliy},
  journal={INFORMS Journal on Computing},
  volume={32},
  number={2},
  pages={507--530},
  year={2020},
  publisher={INFORMS}
}

@article{kong2020,
  title={On the derivation of continuous piecewise linear approximating functions},
  author={Kong, Lingxun and Maravelias, Christos T},
  journal={INFORMS Journal on Computing},
  volume={32},
  number={3},
  pages={531--546},
  year={2020},
  publisher={INFORMS}
}

@article{dehning2020inferring,
  title={Inferring change points in the spread of COVID-19 reveals the effectiveness of interventions},
  author={Dehning, Jonas and Zierenberg, Johannes and Spitzner, F Paul and Wibral, Michael and Neto, Joao Pinheiro and Wilczek, Michael and Priesemann, Viola},
  journal={Science},
  volume={369},
  number={6500},
  pages={eabb9789},
  year={2020},
  publisher={American Association for the Advancement of Science}
}

@article{tobiasz2023multivariate,
  title={Multivariate piecewise linear regression model to predict radiosensitivity using the association with the genome-wide copy number variation},
  author={Tobiasz, Joanna and Al-Harbi, Najla and Bin Judia, Sara and Majid Wakil, Salma and Polanska, Joanna and Alsbeih, Ghazi},
  journal={Frontiers in Oncology},
  volume={13},
  pages={1154222},
  year={2023},
  publisher={Frontiers}
}

@inproceedings{buason2025sample,
  title={Sample-based piecewise linear power flow approximations using second-order sensitivities},
  author={Buason, Paprapee and Misra, Sidhant and Molzahn, Daniel K},
  booktitle={2025 IEEE Kiel PowerTech},
  pages={1--7},
  year={2025},
  organization={IEEE}
}

@inproceedings{ahmadi2013piecewise,
  title={Piecewise linear approximation of generators cost functions using max-affine functions},
  author={Ahmadi, Hamed and Mart{\'\i}, Jos{\'e} R and Moshref, Ali},
  booktitle={2013 IEEE Power \& Energy Society General Meeting},
  pages={1--5},
  year={2013},
  organization={IEEE}
}

@article{gao2018pvc,
  title={Piecewise linear approximation based MILP method for PVC plant planning optimization},
  author={Gao, Xiaoyong and Feng, Zhenhui and Wang, Yuhong and Huang, Xiaolin and Huang, Dexian and Chen, Tao and Lian, Xue},
  journal={Industrial \& Engineering Chemistry Research},
  volume={57},
  number={4},
  pages={1233--1244},
  year={2018},
  publisher={ACS Publications}
}

@article{wagner2002segmented,
  title={Segmented regression analysis of interrupted time series studies in medication use research},
  author={Wagner, Anita K and Soumerai, Stephen B and Zhang, Fang and Ross-Degnan, Dennis},
  journal={Journal of clinical pharmacy and therapeutics},
  volume={27},
  number={4},
  pages={299--309},
  year={2002},
  publisher={Wiley Online Library}
}

@article{bertsimas2007,
  title={Classification and regression via integer optimization},
  author={Bertsimas, Dimitris and Shioda, Romy},
  journal={Operations research},
  volume={55},
  number={2},
  pages={252--271},
  year={2007},
  publisher={Informs}
}

@article{toriello2012fitting,
  title={Fitting piecewise linear continuous functions},
  author={Toriello, Alejandro and Vielma, Juan Pablo},
  journal={European Journal of Operational Research},
  volume={219},
  number={1},
  pages={86--95},
  year={2012},
  publisher={Elsevier}
}

@article{warwicker2022comparison,
  title={A comparison of two mixed-integer linear programs for piecewise linear function fitting},
  author={Warwicker, John Alasdair and Rebennack, Steffen},
  journal={INFORMS Journal on Computing},
  volume={34},
  number={2},
  pages={1042--1047},
  year={2022},
  publisher={INFORMS}
}

@article{cui2018composite,
  title={Composite difference-max programs for modern statistical estimation problems},
  author={Cui, Ying and Pang, Jong-Shi and Sen, Bodhisattva},
  journal={SIAM Journal on Optimization},
  volume={28},
  number={4},
  pages={3344--3374},
  year={2018},
  publisher={SIAM}
}

@article{prokhorov2025change,
  title={Change-point detection in time series using mixed integer programming},
  author={Prokhorov, Artem and Radchenko, Peter and Semenov, Alexander and Skrobotov, Anton},
  journal={Journal of Business \& Economic Statistics},
  pages={1--9},
  year={2025},
  publisher={Taylor \& Francis}
}

@article{harchaoui2010multiple,
  title={Multiple change-point estimation with a total variation penalty},
  author={Harchaoui, Za{\i}d and L{\'e}vy-Leduc, C{\'e}line},
  journal={Journal of the American Statistical Association},
  volume={105},
  number={492},
  pages={1480--1493},
  year={2010},
  publisher={Taylor \& Francis}
}

@article{warwicker2023unified,
  title={A unified framework for bivariate clustering and regression problems via mixed-integer linear programming},
  author={Warwicker, John Alasdair and Rebennack, Steffen},
  journal={Discrete Applied Mathematics},
  volume={336},
  pages={15--36},
  year={2023},
  publisher={Elsevier}
}

@article{ploussard2024piecewise,
  title={Piecewise linear approximation with minimum number of linear segments and minimum error: A fast approach to tighten and warm start the hierarchical mixed integer formulation},
  author={Ploussard, Quentin},
  journal={European Journal of Operational Research},
  volume={315},
  number={1},
  pages={50--62},
  year={2024},
  publisher={Elsevier}
}

@article{zhang2007modified,
  title={A modified Bayes information criterion with applications to the analysis of comparative genomic hybridization data},
  author={Zhang, Nancy R and Siegmund, David O},
  journal={Biometrics},
  volume={63},
  number={1},
  pages={22--32},
  year={2007},
  publisher={Oxford University Press}
}

@article{cao2018multi,
  title={Multi-sensor slope change detection},
  author={Cao, Yang and Xie, Yao and Gebraeel, Nagi},
  journal={Annals of Operations Research},
  volume={263},
  number={1},
  pages={163--189},
  year={2018},
  publisher={Springer}
}

@article{wang2018high,
  title={High dimensional change point estimation via sparse projection},
  author={Wang, Tengyao and Samworth, Richard J},
  journal={Journal of the Royal Statistical Society Series B: Statistical Methodology},
  volume={80},
  number={1},
  pages={57--83},
  year={2018},
  publisher={Oxford University Press}
}

@inproceedings{xie2013sequential,
  title={Sequential multi-sensor change-point detection},
  author={Xie, Yao and Siegmund, David},
  booktitle={2013 Information theory and applications workshop (ITA)},
  pages={1--20},
  year={2013},
  organization={IEEE}
}

@article{bai2003computation,
  title={Computation and analysis of multiple structural change models},
  author={Bai, Jushan and Perron, Pierre},
  journal={Journal of applied econometrics},
  volume={18},
  number={1},
  pages={1--22},
  year={2003},
  publisher={Wiley Online Library}
}

\appendix
\renewcommand{\thesubsection}{A.\arabic{subsection}} 
\renewcommand{\theequation}{A.\arabic{equation}}      

\section*{Appendix}
\addcontentsline{toc}{section}{Appendix}

\section{Choice of Big-\(M\) Constants for PWL Fitting}
\label{sec:bigM}
\numberwithin{equation}{section}

This section derives a data-driven parameter space \(\boldsymbol{\Theta}\) for
\(\boldsymbol{\theta}=(\{m_i,c_i\}_{i=1}^K,\, r_0,\ldots,r_K)\) and derives valid
Big-\(M\) constants for the one-dimensional PWL formulations in
Sections~\ref{subsec:Basic} and~\ref{subsec:Alternate}. Assuming the inputs satisfy \(x_1 < x_2 < \cdots < x_T\), we define the global slope
bounds induced by the data as
\begin{align}
\label{eq:theta_Lm}
L^{m}
&:=
\min_{\substack{t_1,t_2 \in \{1,\ldots,T\} \\ t_1 \neq t_2}}
\frac{y_{t_1}-y_{t_2}}{x_{t_1}-x_{t_2}},\\
\label{eq:theta_Um}
U^{m}
&:=
\max_{\substack{t_1,t_2 \in \{1,\ldots,T\} \\ t_1 \neq t_2}}
\frac{y_{t_1}-y_{t_2}}{x_{t_1}-x_{t_2}} .
\end{align}
These bounds are valid for the slope of any linear segment that fits the data over a sub-interval of the observed domain. Given \(L^m\) and \(U^m\), we define the following intercept bounds such that
each linear function \(m x + c\) can attain observed values over the domain
\(\{x_t\}_{t=1}^T\):
\begin{align}
\label{eq:theta_Lc}
L^{c}
&:=
\min_{t \in \{1,\ldots,T\}}
\min\bigl\{ y_t - L^{m}x_t,\; y_t - U^{m}x_t \bigr\},\\
\label{eq:theta_Uc}
U^{c}
&:=
\max_{t \in \{1,\ldots,T\}}
\max\bigl\{ y_t - L^{m}x_t,\; y_t - U^{m}x_t \bigr\}.
\end{align}
The breakpoint bounds are obtained based on the
convention that \(r_0=x_1\) and \(r_K=x_T\). We restrict all interior breakpoints to
\([x_1,x_T]\) and impose ordering:
\begin{equation}
\label{eq:theta_breaks}
x_1 = r_0 \leq r_1 \leq \cdots \leq r_{K-1} \leq r_K = x_T,
\qquad r_j \in [x_1,x_T]\ \ \forall j=1,\ldots,K-1.
\end{equation}
We therefore define the parameter space:
\begin{equation}
\label{eq:Theta_def}
\boldsymbol{\Theta}
:=
\Bigl\{
\boldsymbol{\theta}=(\{m_i,c_i\}_{i=1}^K, r_0,\ldots,r_K)\ :\
m_i \in [L^m,U^m],\ c_i \in [L^c,U^c]\ \forall i,\ \text{and \eqref{eq:theta_breaks} holds}
\Bigr\}.
\end{equation}
We use the above bounds to derive valid Big-\(M\) constants for the
value-assignment and breakpoint-localization constraints in
\eqref{eq:A2_upper} - \eqref{eq:A2_lower} and
\eqref{eq:A3_right} - \eqref{eq:A3_left}, respectively. To ensure a fair and
consistent comparison across MIP benchmarks, we adopt the same Big-\(M\)
construction as in \citet{rebennack2020}. For the value-assignment constraints
\eqref{eq:A2_upper} - \eqref{eq:A2_lower}, the constant \(M^A_{1,t}\) must
dominate the fitted value at \(x_t\) whenever \(\delta_{j,t}=0\), and is obtained as:
\begin{equation}
\label{eq:MA1_app}
M^A_{1,t}
=
\max \Bigl\{
|L^{m}x_t| + |L^{c}|,\;
|U^{m}x_t| + |L^{c}|,\;
|L^{m}x_t| + |U^{c}|,\;
|U^{m}x_t| + |U^{c}|
\Bigr\},
\qquad t = 1,\ldots,T.
\end{equation}
For the breakpoint-localization constraints
\eqref{eq:A3_right} - \eqref{eq:A3_left}, note that
\(x_1 \le r_j \le x_T\) for all breakpoints \(r_j\), and
\(x_1 \le x_t \le x_T\) for all \(t\). It is therefore sufficient to choose:
\begin{align}
\label{eq:MA2_app}
M^A_{2,t} &=x_t - x_1 ,\\
\label{eq:MA3_app}
M^A_{3,t} &= x_T - x_t,
\end{align}
which ensures that the corresponding inequalities are automatically
satisfied whenever \(\delta_{j,t}=0\). Finally, the linearized continuity constraints
\eqref{eq:B2_1} - \eqref{eq:B2_4} require a Big-\(M\) constant \(M^B_{4,t}\)
that dominates expressions of the form
\(|x_t(m_j - m_{j+1}) - (c_{j+1} - c_j)|\).
Over \((m_j,m_{j+1}) \in [L^m, U^m]^2\) and
\((c_j,c_{j+1}) \in [L^c, U^c]^2\), we have
\(|m_j - m_{j+1}| \le U^m - L^m\) and
\(|c_{j+1} - c_j| \le U^c - L^c\).
Hence, we obtain:
\begin{equation}
\label{eq:MB4_app}
M^B_{4,t}
=
|x_t|\,(U^{m} - L^{m}) + (U^{c} - L^{c}),
\qquad t = 1,\ldots,T.
\end{equation}

\section{MIP for PWL Fitting with $\ell_0$ Regularization}
\label{sec:L0_regularized}
\numberwithin{equation}{section}

As discussed in Section~\ref{sec:formulation_1d}, the Basic, Extended Basic, and Extended Alternative formulations allow the parameter $K$ to act as an upper bound on the number of linear segments. Consequently, these formulations can be naturally adapted to perform piecewise-linear (PWL) fitting with $\ell_0$ regularization on the number of segments. To enable $\ell_0$ regularization, we explicitly count the number of segments that are activated in the model. This is achieved by introducing an auxiliary binary variable $\delta_j^R$ for each segment $j = 1,\ldots,K$, which indicates whether segment $j$ is used. The following constraint block is added to the formulation:

\begin{enumerate}[label=\textbf{(R\arabic*)}, wide]

\item Segment usage indicator constraints:
\begin{equation}
\label{eq:R1_1}
    \sum_{t=1}^T \delta_{j,t} \leq M^R \delta^R_j,
    \qquad \forall j = 1,\ldots,K.
\end{equation}

\end{enumerate}

Here, $M^R$ is a sufficiently large constant. Since $\delta_{j,t} \in \{0,1\}$ for all $t$, a valid choice is $M^R = T$, as $\sum_{t=1}^T \delta_{j,t} \le T$ holds trivially. Constraint~\eqref{eq:R1_1} enforces $\delta_j^R = 1$ whenever segment $j$ is used by at least one data point, and allows $\delta_j^R = 0$ only when the segment is entirely inactive.

\paragraph{Objective function}

With these additional variables, the objective function is modified to penalize the number of active segments:
\begin{equation}
\min \;\; \sum_{t=1}^T \lvert y_t - \hat{y}_t \rvert^{\,q}
\;+\;
\lambda_0 \sum_{j=1}^K \delta^R_j,
\qquad q \in \{1,2\}.
\label{eq:objective}
\end{equation}

Here, $\lambda_0 > 0$ is a regularization parameter that controls the trade-off between goodness of fit and model complexity, and can be selected based on the dataset or by cross-validation.

\section{LP Relaxation: Feasible Construction Supporting Proposition~\ref{prop:projection_subset_Basic}}
\label{sec:prop_continued}
We show that it is possible to construct a feasible solution that satisfies the
breakpoint localization constraints
\eqref{eq:A3_right}, \eqref{eq:A3_left},
\eqref{eq:A3_endpoints}, and \eqref{eq:A3_breakpoints_ordering}.
All remaining constraints of the Basic formulation are trivially satisfiable. We therefore focus on explicitly constructing the breakpoint
variables $\{r_j\}$ that support a feasible LP-relaxed segment assignment. Consider the instance discussed in Proposition~\ref{prop:projection_subset_Basic} with \(T=8\) data points
\(\{x_t\}_{t=1}^T\) and \(K=4\) segments. The LP-relaxed
segment-assignment variables are as follows:
\[
\widetilde{\delta}_{1,\cdot} = [0.5,\,0.5,\,0.5,\,0.5,\,0.5,\,0.5,\,0.5,\,0.5],
\]
\[
\widetilde{\delta}_{2,\cdot} = [0.4,\,0.1,\,0.4,\,0.1,\,0.4,\,0.1,\,0.4,\,0.1],
\]
\[
\widetilde{\delta}_{3,\cdot} = [0.1,\,0.4,\,0.1,\,0.4,\,0.1,\,0.4,\,0.1,\,0.4],
\]
\[
\widetilde{\delta}_{4,\cdot} = [0,\,0,\,0,\,0,\,0,\,0,\,0,\,0].
\]

The big-$M$ constants in constraints \eqref{eq:A3_right} and \eqref{eq:A3_left}
are given by
\[
M_{2,t}^A = x_t - x_1, \qquad
M_{3,t}^A = x_8 - x_t .
\]

Since this is a $K=4$ segment example, we construct $K+1=5$ breakpoints as follows:
\[
r_0 = x_1, \qquad
r_1 = r_2 = r_3 = \frac{x_1 + x_8}{2}, \qquad
r_4 = x_8 .
\]

We now verify feasibility with respect to constraints
\eqref{eq:A3_right} and \eqref{eq:A3_left} for each segment.

\paragraph{Segment $j=1$.}
For $\widetilde{\delta}_{1,t} = 0.5$ for all $t$, constraint \eqref{eq:A3_right}
requires
\[
x_t - r_1 \le M_{2,t}^A (1 - \widetilde{\delta}_{1,t})
= \frac{x_t - x_1}{2}.
\]
This implies
\[
x_t - \frac{x_t - x_1}{2} \le r_1
\quad \Longleftrightarrow \quad
\frac{x_t + x_1}{2} \le r_1,
\]
which is satisfied for all $t$ by the choice
$r_1 = \frac{x_1 + x_8}{2}$.

Constraint \eqref{eq:A3_left} is always satisfied since
\[
r_0 - x_t = x_1 - x_t \le 0 \le M_{3,t}^A (1 - \widetilde{\delta}_{1,t}).
\]

\paragraph{Segments $j=2,3$.}
For $j\in\{2,3\}$, we have $\widetilde{\delta}_{j,t} \in \{0.1,0.4\}$ and
$r_j = r_{j-1} = \frac{x_1 + x_8}{2}$.
Constraint \eqref{eq:A3_right} requires
\[
x_t - r_j
= x_t - \frac{x_1 + x_8}{2}
\le \frac{x_t - x_1}{2}
\le 0.6 (x_t - x_1)
= M_{2,t}^A (1 - \widetilde{\delta}_{j,t}),
\]
which holds since $1 - \widetilde{\delta}_{j,t} \ge 0.6$.

To verify constraint \eqref{eq:A3_left}, we observe that
\[
r_{j-1} - x_t = \frac{x_1 + x_8}{2} - x_t.
\]
We define $s_t := x_t - x_1$ and $L := x_T - x_1$.
Then
\[
r_{j-1} - x_t
= \frac{L}{2} - s_t
\le 0.6 (L - s_t)
= 0.6 (x_T - x_t)
\le M_{3,t}^A (1 - \widetilde{\delta}_{j,t}),
\]
where the inequality $\frac{L}{2} - s_t
\le 0.6 (L - s_t)$ holds since $0.6 (L - s_t) - (\frac{L}{2} - s_t) = 0.1 L + 0.4 s_t \geq 0$, since  $L, s_t \ge 0$. Furthermore, we also have  $1 - \widetilde{\delta}_{j,t} \ge 0.6$.

\paragraph{Segment $j=4$.}
For the final segment, $\widetilde{\delta}_{4,t} = 0$ for all $t$.
Hence, constraints \eqref{eq:A3_right} and \eqref{eq:A3_left} reduce to
\[
x_t - r_4 \le M_{2,t}^A,
\qquad
r_3 - x_t \le M_{3,t}^A,
\]
both of which are trivially satisfied by construction. This completes the feasible LP-relaxed construction supporting
Proposition~\ref{prop:projection_subset_Basic}.

\label{subsec:prop_continued}
\numberwithin{equation}{subsection}

\section{Computational Results for Univariate Piecewise Linear Fitting without Continuity }
\label{sec:OneD_PWL_non_continuous}

\begin{table}[H]
\centering
\scriptsize
\setlength{\tabcolsep}{5pt}
\renewcommand{\arraystretch}{0.5}

\caption{
Comparison of average runtimes (in seconds) for the Basic, Alternative, and Extended formulations for univariate non-continuous PWL fitting, averaged over three random seeds (0, 10000, and 20000). Bold entries mark the fastest formulation, and bracketed values denote the average optimality gap for instances that reached the 2000~s time limit in all three runs.
}
\label{tab:runtime_1D_without_continuity}
\resizebox{\linewidth}{!}{%
\begin{tabular}{c|ccccc|ccccc|ccccc}
\toprule
& \multicolumn{5}{c|}{Basic}
& \multicolumn{5}{c|}{Alternative}
& \multicolumn{5}{c}{Extended} \\
\cmidrule(lr){2-6}\cmidrule(lr){7-11}\cmidrule(lr){12-16}
$T$
& $K=2$ & $K=3$ & $K=4$ & $K=5$ & $K=6$
& $K=2$ & $K=3$ & $K=4$ & $K=5$ & $K=6$
& $K=2$ & $K=3$ & $K=4$ & $K=5$ & $K=6$ \\
\midrule

\multicolumn{16}{c}{\textbf{$\ell_1$ loss}} \\
\midrule
\multicolumn{16}{l}{\textbf{AAPL}} \\
100 & 0.21 & 0.45 & 2.36 & 9.54 & 44.60 & \textbf{0.07} & 0.31 & 1.57 & 7.40 & 43.61 & 0.08 & \textbf{0.27} & \textbf{1.17} & \textbf{6.23} & \textbf{31.14} \\
200 & 0.41 & 1.01 & 7.03 & 58.04 & 548.07 & \textbf{0.17} & 0.80 & \textbf{4.30} & 51.54 & \textbf{528.29} & 0.22 & \textbf{0.67} & 5.78 & \textbf{34.52} & 616.81 \\
300 & 0.78 & 2.22 & 16.34 & 60.58 & \textbf{641.66} & \textbf{0.33} & 2.00 & \textbf{9.80} & 57.53 & 756.24 & 0.35 & \textbf{1.77} & 10.16 & \textbf{45.97} & 747.74 \\
400 & 0.87 & 3.40 & 24.44 & 84.73 & 376.22 & \textbf{0.46} & 3.34 & \textbf{15.03} & 69.83 & 472.83 & 0.56 & \textbf{2.83} & 15.29 & \textbf{51.59} & \textbf{255.81} \\
500 & 2.03 & 5.52 & 43.19 & 160.36 & 799.53 & 0.62 & 4.90 & \textbf{25.98} & 103.47 & \textbf{743.07} & \textbf{0.59} & \textbf{3.95} & 27.77 & \textbf{84.14} & 1234.57 \\

\midrule
\multicolumn{16}{l}{\textbf{AMZN}} \\
100 & 0.18 & 0.29 & 0.91 & 2.69 & 15.17 & \textbf{0.08} & 0.24 & 0.55 & 2.63 & 12.67 & 0.09 & \textbf{0.17} & \textbf{0.50} & \textbf{2.16} & \textbf{8.16} \\
200 & 0.46 & 1.01 & 4.06 & 18.35 & 68.64 & \textbf{0.20} & \textbf{0.71} & \textbf{2.89} & 15.62 & 65.20 & 0.25 & 0.76 & 3.32 & \textbf{11.38} & \textbf{42.54} \\
300 & 0.65 & 2.11 & 20.61 & 109.59 & 485.34 & \textbf{0.27} & 1.48 & 13.18 & 99.77 & \textbf{368.45} & 0.32 & \textbf{1.20} & \textbf{13.09} & \textbf{71.96} & 633.31 \\
400 & 0.83 & 2.97 & 26.16 & 132.86 & 1003.65 & \textbf{0.41} & \textbf{2.37} & 18.88 & 113.63 & 1169.84 & 0.50 & 2.63 & \textbf{17.09} & \textbf{76.77} & \textbf{705.96} \\
500 & 1.23 & 6.87 & 36.90 & 159.57 & 940.47 & \textbf{0.76} & 4.51 & 25.64 & 111.92 & 914.27 & 0.94 & \textbf{3.76} & \textbf{24.54} & \textbf{79.89} & \textbf{812.16} \\

\midrule
\multicolumn{16}{l}{\textbf{GOOGL}} \\
100 & 0.17 & 0.50 & 3.30 & 14.92 & 67.86 & \textbf{0.07} & 0.27 & 1.65 & 11.22 & 63.85 & 0.08 & \textbf{0.25} & \textbf{1.49} & \textbf{10.09} & \textbf{47.05} \\
200 & 0.50 & 0.93 & 6.68 & 34.00 & 209.95 & \textbf{0.17} & \textbf{0.59} & 3.66 & 31.58 & \textbf{182.00} & 0.23 & 0.62 & \textbf{3.49} & \textbf{21.46} & 288.90 \\
300 & 0.82 & 2.03 & 10.65 & 47.78 & \textbf{387.77} & \textbf{0.38} & 1.37 & \textbf{6.49} & 46.49 & 401.53 & 0.43 & \textbf{1.09} & 7.35 & \textbf{33.68} & 503.08 \\
400 & 1.11 & 3.27 & 19.55 & 100.14 & 792.51 & \textbf{0.55} & \textbf{2.23} & 12.50 & 73.91 & 904.57 & 0.68 & 2.77 & \textbf{11.83} & \textbf{58.94} & \textbf{758.00} \\
500 & 2.96 & 4.77 & 33.24 & 132.32 & 979.41 & \textbf{0.82} & \textbf{3.57} & \textbf{14.16} & 106.28 & 964.38 & 1.46 & 4.17 & 15.79 & \textbf{89.27} & \textbf{735.03} \\

\midrule
\multicolumn{16}{l}{\textbf{MSFT}} \\
100 & 0.13 & 0.33 & 0.89 & 4.80 & 15.12 & 0.09 & 0.33 & 0.615 & 2.99 & 14.47 & \textbf{0.08} & \textbf{0.19} & \textbf{0.605} & \textbf{2.71} & \textbf{7.94} \\
200 & 0.40 & 1.18 & 4.42 & 16.87 & 64.59 & \textbf{0.19} & 0.67 & 3.23 & 15.94 & 65.92 & 0.23 & \textbf{0.60} & \textbf{3.03} & \textbf{9.32} & \textbf{42.96} \\
300 & 0.78 & 1.89 & 9.03 & 37.06 & \textbf{134.58} & \textbf{0.27} & \textbf{1.12} & 7.97 & 24.88 & 160.75 & 0.38 & 1.20 & \textbf{5.91} & \textbf{20.55} & 166.90 \\
400 & 1.12 & 3.17 & 19.31 & 84.15 & 349.02 & \textbf{0.34} & 1.82 & 12.34 & 72.91 & 264.26 & 0.80 & \textbf{1.79} & \textbf{11.13} & \textbf{46.77} & \textbf{202.88} \\
500 & 2.01 & 4.86 & 24.63 & 142.25 & \textbf{640.30} & \textbf{0.75} & 3.84 & \textbf{12.74} & 95.67 & 718.71 & 1.17 & \textbf{2.83} & 14.34 & \textbf{71.90} & 682.39 \\

\midrule
\multicolumn{16}{c}{\textbf{$\ell_2$ loss}} \\
\midrule

\multicolumn{16}{l}{\textbf{AAPL}} \\
100 & 0.10 & 0.47 & 1.40 & 4.13 & 17.77 & 0.048 & 0.28 & \textbf{0.76} & 2.84 & 11.92 & \textbf{0.046} & \textbf{0.18} & 0.89 & \textbf{2.68} & \textbf{7.79} \\
200 & 0.51 & 1.26 & 5.35 & 21.36 & 97.85 & 0.17 & 0.49 & \textbf{2.54} & 14.17 & 54.96 & \textbf{0.16} & \textbf{0.34} & 3.73 & \textbf{11.30} & \textbf{45.11} \\
300 & 0.77 & 5.91 & 10.57 & 40.58 & 203.42 & 0.44 & 1.89 & 7.79 & 21.73 & 188.14 & \textbf{0.41} & \textbf{1.66} & \textbf{5.76} & \textbf{16.61} & \textbf{126.76} \\
400 & 1.63 & 7.65 & 20.99 & 35.75 & 144.00 & \textbf{0.42} & 2.46 & 14.66 & 27.14 & 101.07 & 0.54 & \textbf{1.94} & \textbf{11.21} & \textbf{20.62} & \textbf{60.89} \\
500 & 1.45 & 13.18 & 31.91 & 70.50 & 175.84 & \textbf{0.73} & 3.81 & 23.31 & 37.20 & 244.41 & 1.01 & \textbf{3.35} & \textbf{12.21} & \textbf{29.60} & \textbf{153.66} \\

\midrule
\multicolumn{16}{l}{\textbf{AMZN}} \\
100 & 0.12 & 0.37 & 0.74 & 1.57 & 3.36 & \textbf{0.05} & 0.19 & \textbf{0.36} & 0.87 & 2.17 & 0.06 & \textbf{0.14} & 0.39 & \textbf{0.85} & \textbf{1.74} \\
200 & 0.42 & 1.28 & 3.91 & 9.27 & 33.97 & \textbf{0.162} & \textbf{0.80} & \textbf{2.18} & \textbf{5.47} & 19.01 & 0.164 & 0.98 & 2.34 & 6.47 & \textbf{9.78} \\
300 & 0.88 & 3.78 & 13.07 & 37.46 & 100.59 & \textbf{0.32} & \textbf{1.27} & \textbf{5.75} & 38.05 & 65.46 & 0.36 & 1.31 & 7.28 & \textbf{17.61} & \textbf{46.74} \\
400 & 2.85 & 7.32 & 34.62 & 45.92 & 258.19 & 0.44 & 2.52 & 12.08 & 37.07 & 233.32 & \textbf{0.35} & \textbf{1.28} & \textbf{10.57} & \textbf{29.65} & \textbf{100.25} \\
500 & 1.61 & 14.85 & 27.14 & 75.44 & 316.49 & 0.45 & 3.95 & 23.29 & 65.63 & 218.32 & \textbf{0.39} & \textbf{2.48} & \textbf{14.05} & \textbf{33.30} & \textbf{117.59} \\

\midrule
\multicolumn{16}{l}{\textbf{GOOGL}} \\
100 & 0.19 & 0.37 & 1.56 & 5.02 & 14.46 & \textbf{0.05} & 0.25 & 1.01 & 4.85 & 10.21 & 0.06 & \textbf{0.19} & \textbf{0.78} & \textbf{3.87} & \textbf{7.83} \\
200 & 0.24 & 2.50 & 5.21 & 13.60 & 52.28 & 0.084 & 0.86 & 2.47 & 8.33 & 25.18 & \textbf{0.081} & \textbf{0.37} & \textbf{1.71} & \textbf{7.38} & \textbf{13.85} \\
300 & 0.98 & 4.40 & 9.47 & 34.76 & 130.88 & 0.25 & 2.34 & 4.79 & 20.16 & 73.51 & \textbf{0.24} & \textbf{0.84} & \textbf{3.69} & \textbf{11.70} & \textbf{30.15} \\
400 & 1.95 & 5.18 & 14.65 & 48.81 & 172.11 & 0.52 & 5.24 & 14.99 & \textbf{31.53} & 160.02 & \textbf{0.38} & \textbf{1.89} & \textbf{11.29} & 33.80 & \textbf{80.45} \\
500 & 2.31 & 14.12 & 22.45 & 68.87 & 330.60 & 0.57 & \textbf{2.50} & \textbf{10.18} & 42.15 & 226.38 & \textbf{0.55} & 3.63 & 16.92 & \textbf{40.61} & \textbf{161.69} \\

\midrule
\multicolumn{16}{l}{\textbf{MSFT}} \\
100 & 0.10 & 0.30 & 0.60 & 2.73 & 6.04 & 0.05 & 0.130 & 0.45 & \textbf{1.16} & 4.51 & \textbf{0.04} & \textbf{0.125} & \textbf{0.41} & 1.41 & \textbf{3.48} \\
200 & 0.20 & 1.06 & 3.59 & 6.29 & 19.35 & 0.22 & \textbf{0.48} & 1.49 & 5.75 & 13.16 & \textbf{0.12} & 0.51 & \textbf{0.94} & \textbf{4.15} & \textbf{8.26} \\
300 & 1.42 & 3.08 & 9.49 & 21.26 & 61.91 & \textbf{0.26} & \textbf{0.66} & 6.16 & 13.07 & 27.26 & 0.87 & 0.71 & \textbf{5.37} & \textbf{9.62} & \textbf{17.76} \\
400 & 1.09 & 9.35 & 24.89 & 82.91 & 148.32 & 0.87 & 3.06 & 7.10 & 32.18 & 105.60 & \textbf{0.27} & \textbf{1.44} & \textbf{5.42} & \textbf{19.33} & \textbf{39.55} \\
500 & 2.83 & 13.63 & 32.45 & 49.71 & 322.53 & \textbf{0.54} & 12.88 & \textbf{9.03} & 43.44 & 196.35 & 1.29 & \textbf{2.26} & 13.26 & \textbf{40.64} & \textbf{90.30} \\

\bottomrule
\end{tabular}
}

\end{table}

\section{Computational Results for Univariate Piecewise Linear Fitting with Continuity }
\label{sec:OneD_PWL_continuous}

\begin{table}[H]
\centering
\scriptsize
\setlength{\tabcolsep}{5pt}
\renewcommand{\arraystretch}{0.5}

\caption{
Comparison of average runtimes (in seconds) for the Basic, Alternative, Extended Basic, and Extended Alternative formulations for univariate PWL fitting with continuity, averaged over three random seeds (0, 10000, and 20000). Bold entries mark the fastest formulation, and bracketed values denote the average optimality gap for instances that reached the 2000~s time limit in all three runs.
}
\label{tab:runtime_1D_with_continuity}

\resizebox{\linewidth}{!}{%
\begin{tabular}{c|cccc|cccc|cccc|cccc}
\toprule
\multirow{2}{*}{\textbf{$T$}}
& \multicolumn{4}{c|}{\textbf{Basic}}
& \multicolumn{4}{c|}{\textbf{Alternative}}
& \multicolumn{4}{c|}{\textbf{Extended Basic}}
& \multicolumn{4}{c}{\textbf{Extended Alternative}} \\
\cmidrule(lr){2-5} \cmidrule(lr){6-9}
\cmidrule(lr){10-13} \cmidrule(lr){14-17}
& $K{=}2$ & $K{=}3$ & $K{=}4$ & $K{=}5$
& $K{=}2$ & $K{=}3$ & $K{=}4$ & $K{=}5$
& $K{=}2$ & $K{=}3$ & $K{=}4$ & $K{=}5$
& $K{=}2$ & $K{=}3$ & $K{=}4$ & $K{=}5$ \\
\midrule
\multicolumn{17}{c}{\textbf{$\ell_1$ loss}} \\
\midrule

\multicolumn{17}{l}{\textbf{AAPL}} \\
100 & 0.30 & \textbf{1.07} & \textbf{6.23} & 130.65
    & 0.142 & 1.32 & 12.48 & 211.11
    & 0.32 & 1.35 & 6.81 & \textbf{112.66}
    & \textbf{0.141} & 1.30 & 11.25 & 158.75 \\
120 & 0.45 & \textbf{1.95} & 28.36 & 105.24
    & 0.191 & 4.19 & 38.94 & 170.36
    & 0.39 & 2.53 & \textbf{21.34} & \textbf{91.38}
    & \textbf{0.187} & 3.39 & 35.35 & 139.79 \\
140 & 0.44 & \textbf{4.86} & 54.42 & 260.91
    & 0.21 & 5.33 & 60.51 & 429.77
    & 0.37 & 5.48 & \textbf{35.51} & \textbf{196.97}
    & \textbf{0.20} & 7.61 & 58.66 & 314.14 \\
160 & 0.77 & \textbf{1.88} & 334.24 & 1796.27
    & 0.30 & 2.14 & 140.85 & 1335.06
    & 0.58 & 2.00 & \textbf{79.58} & \textbf{664.17}
    & \textbf{0.27} & 2.08 & 132.88 & 1166.65 \\
180 & 0.41 & \textbf{3.74} & 51.75 & [1.13\%]
    & \textbf{0.29} & 5.76 & 95.44 & [4.48\%]
    & 0.76 & 3.83 & \textbf{43.22} & \textbf{1688.45}
    & 0.31 & 4.15 & 90.83 & [2.55\%] \\
\midrule

\multicolumn{17}{l}{\textbf{AMZN}} \\
100 & 0.32 & \textbf{0.98} & \textbf{5.57} & 23.97
    & 0.18 & 1.35 & 14.65 & 88.97
    & 0.42 & 1.37 & 7.26 & \textbf{22.56}
    & \textbf{0.16} & 1.26 & 12.08 & 72.04 \\
120 & 0.29 & \textbf{1.78} & 15.16 & 38.46
    & 0.24 & 2.75 & 44.25 & 143.72
    & 0.58 & 1.84 & \textbf{14.56} & \textbf{31.47}
    & \textbf{0.23} & 2.70 & 32.85 & 129.49 \\
140 & 0.56 & \textbf{2.16} & 41.45 & \textbf{105.66}
    & 0.28 & 3.34 & 79.38 & 909.63
    & 0.41 & 2.42 & \textbf{23.03} & 131.10
    & \textbf{0.25} & 4.22 & 79.18 & 624.09 \\
160 & 0.48 & \textbf{4.16} & 247.68 & 823.29
    & 0.37 & 5.72 & 125.73 & 1825.19
    & 0.70 & 4.68 & \textbf{59.02} & \textbf{719.11}
    & \textbf{0.32} & 4.31 & 131.58 & 1673.22 \\
180 & 0.62 & 28.60 & 113.24 & 1146.75
    & 0.51 & \textbf{9.05} & 92.12 & 1803.92
    & 0.99 & 21.77 & \textbf{47.50} & \textbf{551.23}
    & \textbf{0.48} & 11.12 & 101.55 & 1155.67 \\
\midrule

\multicolumn{17}{l}{\textbf{GOOGL}} \\
100 & 0.28 & \textbf{1.69} & \textbf{15.89} & 165.22
    & 0.16 & 2.06 & 34.62 & 318.88
    & 0.31 & 1.77 & 18.08 & \textbf{157.83}
    & \textbf{0.15} & 1.80 & 30.10 & 278.66 \\
120 & 0.32 & \textbf{1.79} & \textbf{19.47} & 240.13
    & \textbf{0.207} & 2.36 & 43.90 & 468.89
    & 0.29 & 2.38 & 22.87 & \textbf{209.49}
    & 0.209 & 2.12 & 34.81 & 405.41 \\
140 & 0.63 & \textbf{0.93} & \textbf{15.24} & 168.82
    & 0.27 & 1.44 & 37.57 & 494.23
    & 0.53 & 1.62 & 18.46 & \textbf{160.38}
    & \textbf{0.24} & 1.28 & 31.16 & 434.24 \\
160 & 0.59 & \textbf{2.40} & \textbf{13.21} & \textbf{127.05}
    & \textbf{0.30} & 2.61 & 32.22 & 398.68
    & 0.86 & 2.81 & 14.96 & 149.04
    & 0.31 & 2.59 & 30.63 & 365.86 \\
180 & 0.66 & \textbf{1.53} & \textbf{17.82} & 687.04
    & \textbf{0.41} & 2.43 & 36.84 & 931.72
    & 0.64 & 2.61 & 18.26 & \textbf{408.79}
    & 0.42 & 1.86 & 43.18 & 590.02 \\
\midrule

\multicolumn{17}{l}{\textbf{JNJ}} \\
100 & 0.28 & \textbf{2.20} & \textbf{16.11} & 302.37
    & \textbf{0.16} & 3.14 & 29.00 & 484.44
    & 0.35 & 2.52 & 17.49 & 404.73
    & 0.17 & 2.29 & 25.03 & \textbf{301.57} \\
120 & 0.39 & 4.04 & 51.87 & \textbf{146.45}
    & 0.22 & 6.12 & 79.35 & 355.58
    & 0.43 & 3.27 & \textbf{42.56} & 162.83
    & \textbf{0.20} & \textbf{3.22} & 66.79 & 249.36 \\
140 & 0.57 & \textbf{9.29} & \textbf{38.66} & \textbf{175.26}
    & 0.27 & 9.90 & 83.60 & 470.97
    & 0.62 & 9.90 & 48.33 & 232.57
    & \textbf{0.22} & 9.55 & 78.96 & 355.60 \\
160 & 0.57 & 8.58 & 137.80 & 638.75
    & 0.35 & 9.77 & 168.47 & 1264.17
    & 0.45 & \textbf{7.53} & \textbf{96.96} & \textbf{467.80}
    & \textbf{0.33} & 8.33 & 169.21 & 1006.57 \\
180 & 0.77 & \textbf{3.03} & \textbf{32.84} & \textbf{1054.38}
    & 0.38 & 5.47 & 101.47 & 1923.38
    & 0.76 & 5.16 & 39.88 & 1122.53
    & \textbf{0.37} & 3.99 & 98.07 & 1900.39 \\
\midrule

\multicolumn{17}{l}{\textbf{MSFT}} \\
100 & 0.29 & \textbf{0.889} & 12.11 & \textbf{99.89}
    & 0.14 & 0.891 & 23.85 & 216.64
    & 0.30 & 1.55 & \textbf{11.10} & 101.48
    & \textbf{0.13} & 0.91 & 18.43 & 159.53 \\
120 & 0.50 & \textbf{0.96} & \textbf{11.10} & \textbf{95.33}
    & 0.24 & 1.34 & 27.23 & 268.36
    & 0.36 & 1.39 & 11.26 & 111.52
    & \textbf{0.22} & 1.07 & 23.73 & 193.20 \\
140 & 0.52 & \textbf{1.12} & 27.35 & 453.92
    & 0.36 & 1.51 & 46.19 & 671.58
    & 0.65 & 1.74 & \textbf{18.72} & \textbf{204.85}
    & \textbf{0.34} & 1.60 & 28.62 & 543.30 \\
160 & 0.82 & \textbf{1.75} & \textbf{14.47} & 181.86
    & \textbf{0.37} & 4.56 & 36.34 & 592.57
    & 0.73 & 2.04 & 16.64 & \textbf{176.32}
    & 0.42 & 2.73 & 46.37 & 636.00 \\
180 & 0.75 & \textbf{1.86} & \textbf{33.22} & 1449.20
    & \textbf{0.50} & 4.14 & 147.86 & 1354.86
    & 0.78 & 2.46 & 37.61 & \textbf{409.85}
    & 0.53 & 2.12 & 122.80 & 1420.65 \\

\midrule
\multicolumn{17}{c}{\textbf{$\ell_2$ loss}} \\
\midrule

\multicolumn{17}{l}{\textbf{AAPL}} \\
100 & 0.27 & 2.60 & 22.47 & 161.99
    & \textbf{0.128} & 1.15 & 5.71 & 82.62
    & 0.36 & 1.44 & 8.66 & 90.82
    & 0.133 & \textbf{1.00} & \textbf{3.23} & \textbf{77.46} \\
120 & 0.38 & 4.44 & 15.95 & 190.03
    & 0.154 & 1.75 & 11.27 & 75.99
    & 0.35 & 4.11 & 27.83 & 93.04
    & \textbf{0.151} & \textbf{1.23} & \textbf{10.13} & \textbf{51.80} \\
140 & 0.54 & 8.72 & 100.38 & 262.02
    & \textbf{0.14} & 3.11 & \textbf{69.98} & 228.97
    & 0.52 & 9.77 & 342.95 & 333.51
    & 0.15 & \textbf{2.71} & 77.84 & \textbf{209.79} \\
160 & 0.58 & 8.39 & 195.12 & 1199.45
    & 0.33 & 1.93 & 75.16 & 800.37
    & 0.77 & 4.43 & 109.76 & 1625.31
    & \textbf{0.32} & \textbf{1.52} & \textbf{62.14} & \textbf{734.84} \\
180 & 0.73 & 37.92 & 61.52 & 1107.10
    & 0.24 & 3.95 & 19.05 & 532.60
    & 0.83 & 8.23 & 20.71 & 1247.54
    & \textbf{0.23} & \textbf{3.70} & \textbf{16.74} & \textbf{328.27} \\
\midrule

\multicolumn{17}{l}{\textbf{AMZN}} \\
100 & 0.26 & 3.64 & 6.64 & 8.52
    & 0.15 & 1.26 & 10.27 & 17.64
    & 0.38 & 1.73 & \textbf{5.88} & \textbf{6.52}
    & \textbf{0.13} & \textbf{1.21} & 6.39 & 9.81 \\
120 & 0.37 & 2.93 & 172.44 & 13.87
    & 0.194 & \textbf{1.44} & 21.04 & 27.84
    & 0.54 & 2.86 & 12.13 & \textbf{10.30}
    & \textbf{0.187} & 3.60 & \textbf{9.70} & 19.16 \\
140 & 0.43 & 4.08 & 38.48 & 657.72
    & \textbf{0.17} & \textbf{2.84} & 44.96 & 266.12
    & 0.50 & 3.92 & \textbf{18.70} & \textbf{104.99}
    & 0.18 & 6.10 & 35.28 & 111.92 \\
160 & 0.74 & 14.47 & 861.91 & 488.56
    & \textbf{0.289} & 3.54 & 52.49 & 351.14
    & 0.93 & 7.72 & 66.42 & 479.91
    & 0.293 & \textbf{3.09} & \textbf{51.18} & \textbf{230.56} \\
180 & 0.85 & 25.57 & 735.42 & 1582.37
    & 0.44 & \textbf{8.94} & 53.06 & \textbf{483.64}
    & 1.28 & 12.47 & 248.82 & 1124.47
    & \textbf{0.43} & 8.97 & \textbf{49.11} & 641.01 \\
\midrule

\multicolumn{17}{l}{\textbf{GOOGL}} \\
100 & 0.30 & 3.42 & 44.61 & 768.49
    & 0.14 & 1.46 & 31.12 & 231.75
    & 0.35 & 3.13 & 40.20 & 464.61
    & \textbf{0.13} & \textbf{0.91} & \textbf{19.42} & \textbf{189.18} \\
120 & 0.31 & 7.44 & 84.70 & 1251.68
    & 0.19 & 1.78 & 27.71 & 318.52
    & 0.38 & 5.80 & 52.89 & 632.04
    & \textbf{0.15} & \textbf{1.49} & \textbf{26.08} & \textbf{248.68} \\
140 & 0.34 & 1.69 & 12.45 & 188.17
    & 0.34 & 1.21 & \textbf{7.32} & 156.89
    & \textbf{0.31} & 1.48 & 7.81 & \textbf{36.78}
    & 0.38 & \textbf{0.81} & 9.54 & 138.48 \\
160 & 0.49 & 1.78 & 15.19 & 411.29
    & 0.305 & 1.32 & \textbf{6.39} & 266.81
    & 0.64 & 1.61 & 10.62 & \textbf{65.75}
    & \textbf{0.302} & \textbf{0.90} & 10.17 & 191.90 \\
180 & 1.40 & 1.69 & 33.33 & [0.06\%]
    & 0.395 & 2.15 & 11.01 & \textbf{84.86}
    & 1.13 & 2.25 & 13.32 & 217.78
    & \textbf{0.391} & \textbf{0.73} & \textbf{10.76} & 144.43 \\
\midrule

\multicolumn{17}{l}{\textbf{JNJ}} \\
100 & 0.27 & 8.57 & 53.53 & 287.90
    & \textbf{0.146} & 1.40 & 25.14 & \textbf{155.16}
    & 0.40 & 4.53 & 44.05 & 175.83
    & 0.154 & \textbf{1.24} & \textbf{20.95} & 166.25 \\
120 & 0.39 & 6.62 & 148.51 & 1100.91
    & \textbf{0.172} & 1.84 & 68.98 & 214.29
    & 0.64 & 5.07 & 87.64 & \textbf{154.92}
    & 0.174 & \textbf{1.52} & \textbf{51.94} & 200.63 \\
140 & 0.54 & 9.90 & 262.80 & 607.15
    & 0.224 & \textbf{4.46} & 52.01 & \textbf{263.09}
    & 0.55 & 12.73 & 85.49 & 298.88
    & \textbf{0.222} & 4.66 & \textbf{44.77} & 417.90 \\
160 & 0.73 & 11.93 & 717.23 & 1949.38
    & 0.220 & \textbf{4.51} & 81.35 & \textbf{369.85}
    & 0.69 & 10.44 & 143.45 & 1030.99
    & \textbf{0.217} & 5.39 & \textbf{77.18} & 466.71 \\
180 & 1.20 & 10.27 & 309.29 & 988.06
    & 0.229 & \textbf{2.43} & 27.32 & 1024.29
    & 0.61 & 7.72 & \textbf{24.95} & 1319.42
    & \textbf{0.228} & 2.71 & 29.08 & \textbf{450.44} \\
\midrule

\multicolumn{17}{l}{\textbf{MSFT}} \\
100 & 0.45 & 1.34 & 12.34 & 238.72
    & \textbf{0.13} & \textbf{0.76} & 9.74 & 67.96
    & 0.48 & 1.98 & 15.97 & 140.18
    & 0.14 & 0.77 & \textbf{3.79} & \textbf{24.11} \\
120 & 0.39 & 3.76 & 10.61 & 494.53
    & \textbf{0.19} & \textbf{0.54} & 8.73 & 67.36
    & 0.55 & 1.65 & 10.76 & 138.72
    & 0.20 & 0.63 & \textbf{6.22} & \textbf{42.53} \\
140 & 0.72 & 3.65 & 20.34 & 408.75
    & 0.28 & \textbf{0.92} & 19.13 & 140.22
    & 0.70 & 1.86 & 142.44 & 743.85
    & \textbf{0.27} & 1.43 & \textbf{7.15} & \textbf{56.94} \\
160 & 0.81 & 2.48 & 106.90 & 738.98
    & \textbf{0.307} & 3.40 & \textbf{27.60} & 175.99
    & 0.84 & 2.74 & 59.75 & 241.66
    & 0.314 & \textbf{0.83} & 29.92 & \textbf{154.74} \\
180 & 1.19 & 5.20 & 55.13 & 1424.02
    & 0.58 & 2.41 & 48.27 & \textbf{658.91}
    & 1.39 & 4.96 & 70.93 & 1062.43
    & \textbf{0.57} & \textbf{1.46} & \textbf{41.70} & 1445.35 \\

\bottomrule
\end{tabular}%
}

\end{table}

\section{Computational Results for Multi-dimensional Multiple Change-point Detection}
\label{sec:MultiD_CPD}

\begin{table}[H]
\centering
\scriptsize
\setlength{\tabcolsep}{5pt}
\renewcommand{\arraystretch}{0.5}

\caption{
Comparison of average runtimes (in seconds) for the Basic, Alternative, and Extended formulations for multi-dimensional change-point detection, averaged over three random seeds (0, 10000, and 20000). Bold entries mark the fastest formulation, and bracketed values denote the average optimality gap for instances that reached the 2000~s time limit in all three runs.
}
\label{tab:runtime_MultiD_without_continuity}
\resizebox{\linewidth}{!}{%
\begin{tabular}{c|ccccc|ccccc|ccccc}
\toprule
\multirow{2}{*}{\textbf{$T$}}
& \multicolumn{5}{c|}{\textbf{Basic}}
& \multicolumn{5}{c|}{\textbf{Alternative}}
& \multicolumn{5}{c}{\textbf{Extended}} \\
\cmidrule(lr){2-6} \cmidrule(lr){7-11} \cmidrule(lr){12-16}
& $D{=}2$ & $D{=}4$ & $D{=}6$ & $D{=}8$ & $D{=}10$
& $D{=}2$ & $D{=}4$ & $D{=}6$ & $D{=}8$ & $D{=}10$
& $D{=}2$ & $D{=}4$ & $D{=}6$ & $D{=}8$ & $D{=}10$ \\
\midrule
\multicolumn{16}{c}{\textbf{$\ell_1$ loss}} \\
\midrule
\multicolumn{16}{l}{\textbf{$K=2$}} \\
100 & 0.32 & 0.42 & 0.84 & 0.66 & 1.06 & \textbf{0.127} & \textbf{0.26} & \textbf{0.43} & \textbf{0.54} & \textbf{0.71} & 0.133 & 0.30 & 0.47 & 0.62 & 0.77 \\
200 & 0.75 & 1.28 & 1.80 & 2.31 & 2.35 & 0.46 & 0.86 & \textbf{0.94} & 1.41 & \textbf{1.51} & \textbf{0.34} & \textbf{0.76} & 1.10 & \textbf{1.35} & 1.68 \\
300 & 1.04 & 1.88 & 2.88 & 3.11 & 6.32 & 0.70 & 1.18 & \textbf{1.72} & \textbf{2.39} & \textbf{3.31} & \textbf{0.63} & \textbf{1.00} & 2.68 & 3.40 & 3.33 \\
400 & 1.28 & 3.46 & 3.77 & 7.17 & 8.28 & \textbf{1.05} & 1.95 & 3.70 & \textbf{5.32} & \textbf{5.42} & 1.42 & \textbf{1.63} & \textbf{3.53} & 5.93 & 5.70 \\
500 & 2.08 & 5.05 & 7.06 & 12.85 & 17.06 & 1.09 & \textbf{2.88} & 5.95 & \textbf{7.70} & \textbf{9.29} & \textbf{1.02} & 3.15 & \textbf{5.64} & 9.59 & 11.82 \\
\midrule
\multicolumn{16}{l}{\textbf{$K=3$}} \\
100 & 0.61 & 1.00 & 1.93 & 2.59 & 3.47 & \textbf{0.39} & 0.99 & 1.46 & 2.18 & \textbf{2.91} & 0.45 & \textbf{0.79} & \textbf{1.36} & \textbf{2.10} & 3.01 \\
200 & 1.77 & 3.82 & 4.96 & 6.51 & \textbf{9.20} & 1.43 & 3.34 & 4.20 & 5.64 & 9.33 & \textbf{1.37} & \textbf{2.86} & \textbf{3.93} & \textbf{5.25} & 9.35 \\
300 & 3.08 & 7.06 & \textbf{11.18} & 14.77 & 20.48 & 2.88 & \textbf{6.06} & 12.01 & 12.56 & \textbf{17.34} & \textbf{2.47} & 6.43 & 14.67 & \textbf{10.15} & 26.82 \\
400 & 8.39 & 13.47 & 23.30 & 25.55 & 31.73 & \textbf{4.01} & \textbf{12.35} & 19.11 & 20.69 & 39.84 & 4.84 & 13.16 & \textbf{18.76} & \textbf{20.18} & \textbf{22.89} \\
500 & 11.77 & 21.04 & 44.61 & 54.40 & 94.15 & 8.54 & \textbf{18.09} & 38.76 & 58.94 & \textbf{74.34} & \textbf{7.96} & 20.44 & \textbf{34.25} & \textbf{47.36} & 75.25 \\
\midrule
\multicolumn{16}{l}{\textbf{$K=4$}} \\
100 & 1.97 & 3.36 & 12.93 & 12.43 & 24.83 & 1.48 & 3.18 & 9.57 & 14.51 & 24.01 & \textbf{1.37} & \textbf{3.08} & \textbf{9.18} & \textbf{11.90} & \textbf{23.41} \\
200 & 12.19 & 24.36 & 34.62 & 53.38 & 59.43 & 10.19 & 20.62 & 32.13 & 43.80 & 56.04 & \textbf{8.19} & \textbf{17.07} & \textbf{26.44} & \textbf{38.52} & \textbf{53.75} \\
300 & 21.21 & 59.65 & 87.13 & 111.66 & 145.42 & \textbf{18.47} & \textbf{43.69} & 73.24 & 121.48 & 166.40 & 18.62 & 49.85 & \textbf{63.37} & \textbf{99.51} & \textbf{135.73} \\
400 & 43.09 & 93.00 & 177.13 & 224.56 & 290.49 & 35.27 & 76.54 & 147.84 & 183.71 & 387.74 & \textbf{33.65} & \textbf{68.79} & \textbf{136.67} & \textbf{159.59} & \textbf{289.92} \\
500 & 68.40 & 182.21 & 380.61 & 273.11 & 943.27 & 40.58 & 141.87 & 386.88 & 272.49 & \textbf{348.27} & \textbf{36.92} & \textbf{136.98} & \textbf{316.29} & \textbf{233.58} & 365.28 \\
\midrule
\multicolumn{16}{l}{\textbf{$K=5$}} \\
100 & 8.46 & \textbf{23.57} & 71.18 & 84.39 & \textbf{178.05} & 7.25 & 23.69 & 75.07 & 106.36 & 207.97 & \textbf{5.99} & 23.88 & \textbf{59.56} & \textbf{78.70} & 181.95 \\
200 & 43.69 & 73.02 & 164.74 & 244.66 & \textbf{312.46} & 38.82 & 74.15 & 160.91 & 256.74 & 403.08 & \textbf{24.26} & \textbf{61.08} & \textbf{123.08} & \textbf{195.60} & 313.12 \\
300 & 107.71 & 306.62 & 370.04 & 512.63 & 775.36 & 107.85 & 297.65 & 383.43 & 597.56 & 1318.19 & \textbf{97.27} & \textbf{265.70} & \textbf{321.02} & \textbf{424.71} & \textbf{746.50} \\
400 & 132.11 & 436.38 & 1125.55 & 1690.91 & [7.14\%] & 132.00 & 459.78 & \textbf{999.04} & 1839.47 & [12.23\%] & \textbf{109.27} & \textbf{376.67} & 1022.17 & \textbf{1406.80} & \textbf{[5.84\%]} \\
500 & 220.75 & 841.12 & [8.79\%] & [6.57\%] & [9.08\%] & 212.20 & 743.94 & [5.54\%] & [11.35\%] & [10.72\%] & \textbf{167.31} & \textbf{666.43} & \textbf{1977.82} & \textbf{1985.97} & \textbf{1895.33} \\
\midrule
\multicolumn{16}{c}{\textbf{$\ell_2$ loss}} \\
\midrule

\multicolumn{16}{l}{\textbf{$K=2$}} \\
100 & 0.17 & 0.35 & 0.52 & \textbf{0.55} & 0.82 & 0.16 & 0.26 & \textbf{0.36} & 0.67 & 0.78 & \textbf{0.11} & \textbf{0.21} & 0.42 & 0.58 & \textbf{0.75} \\
200 & 0.53 & 1.30 & 1.93 & 4.73 & 5.78 & 0.27 & 0.95 & \textbf{1.33} & \textbf{1.78} & 3.75 & \textbf{0.26} & \textbf{0.80} & 1.36 & 2.16 & \textbf{2.66} \\
300 & 1.91 & 4.73 & 10.57 & 12.21 & 25.62 & 2.26 & 2.14 & \textbf{2.85} & \textbf{4.44} & \textbf{4.14} & \textbf{0.79} & \textbf{1.53} & 3.08 & 5.54 & 6.01 \\
400 & 4.96 & 8.96 & 9.91 & 7.29 & 15.71 & \textbf{0.84} & \textbf{2.81} & \textbf{5.81} & \textbf{6.28} & 10.63 & 1.24 & 4.11 & 7.11 & 20.84 & \textbf{9.15} \\
500 & 8.23 & 11.27 & 33.28 & 48.92 & 15.12 & 1.92 & 6.12 & 11.43 & \textbf{12.40} & 15.90 & \textbf{1.49} & \textbf{5.53} & \textbf{8.23} & 12.63 & \textbf{14.82} \\
\midrule

\multicolumn{16}{l}{\textbf{$K=3$}} \\
100 & 0.56 & 1.27 & 2.05 & 2.38 & 2.96 & 0.32 & 0.92 & 1.60 & \textbf{2.20} & \textbf{2.30} & \textbf{0.28} & \textbf{0.73} & \textbf{1.19} & 2.97 & 3.08 \\
200 & 1.72 & 4.81 & 8.62 & 13.33 & 12.96 & \textbf{0.98} & \textbf{3.06} & 8.91 & 15.85 & \textbf{9.98} & 2.12 & 4.97 & \textbf{6.69} & \textbf{9.90} & 12.52 \\
300 & 6.34 & 17.01 & 27.22 & 25.24 & 31.27 & \textbf{2.44} & \textbf{7.64} & \textbf{10.41} & \textbf{14.79} & \textbf{20.65} & 2.78 & 9.97 & 12.78 & 16.46 & 23.82 \\
400 & 16.29 & 60.06 & 56.94 & 56.84 & 47.15 & \textbf{6.19} & 29.87 & \textbf{22.23} & \textbf{32.57} & 67.64 & 9.44 & \textbf{12.14} & 29.38 & 38.00 & \textbf{36.58} \\
500 & 43.11 & 33.12 & 116.81 & 118.75 & 170.05 & 11.63 & \textbf{19.37} & 57.30 & 71.00 & \textbf{104.29} & \textbf{6.55} & 28.09 & \textbf{45.33} & \textbf{61.92} & 105.11 \\
\midrule

\multicolumn{16}{l}{\textbf{$K=4$}} \\
100 & 1.45 & 2.41 & 7.51 & 6.42 & 12.33 & \textbf{0.72} & 1.79 & \textbf{4.45} & 7.56 & \textbf{12.22} & 0.74 & \textbf{1.78} & 4.82 & \textbf{5.94} & 12.55 \\
200 & 6.51 & 15.61 & 26.90 & 37.81 & 61.64 & 6.55 & 18.11 & \textbf{21.71} & \textbf{30.45} & \textbf{41.45} & \textbf{2.40} & \textbf{7.49} & 22.64 & 32.78 & 45.38 \\
300 & 17.36 & 52.25 & 89.15 & 107.34 & 116.63 & \textbf{10.42} & 43.28 & 84.61 & 110.78 & 114.13 & 16.64 & \textbf{31.35} & \textbf{75.83} & \textbf{59.60} & \textbf{103.09} \\
400 & 47.97 & 82.21 & \textbf{129.21} & \textbf{134.86} & \textbf{199.43} & 33.13 & 62.74 & 145.38 & 180.83 & 267.45 & \textbf{23.67} & \textbf{53.20} & 134.96 & 178.73 & 301.15 \\
500 & \textbf{39.43} & 149.63 & \textbf{162.11} & 292.20 & 945.36 & 57.98 & \textbf{111.22} & 218.01 & \textbf{202.94} & 301.30 & 41.80 & 114.92 & 309.12 & 322.91 & \textbf{245.40} \\
\midrule

\multicolumn{16}{l}{\textbf{$K=5$}} \\
100 & 4.05 & 10.43 & 19.88 & \textbf{26.67} & \textbf{56.58} & 3.76 & 6.43 & 18.93 & 27.86 & 65.64 & \textbf{2.57} & \textbf{6.03} & \textbf{18.61} & 28.57 & 61.69 \\
200 & 21.36 & 27.44 & 59.87 & 75.27 & 141.25 & 12.54 & 22.55 & 63.98 & 78.97 & 164.19 & \textbf{9.98} & \textbf{21.19} & \textbf{54.11} & \textbf{70.63} & \textbf{140.31} \\
300 & 47.56 & 119.88 & \textbf{179.46} & 280.13 & 427.74 & 36.57 & 99.12 & 258.05 & 261.51 & \textbf{302.62} & \textbf{33.61} & \textbf{84.77} & 189.13 & \textbf{239.71} & 334.51 \\
400 & 89.21 & 303.21 & 629.56 & 1038.23 & 1189.59 & 52.26 & 186.45 & \textbf{553.74} & \textbf{855.88} & 1103.48 & \textbf{45.98} & \textbf{169.28} & 782.24 & 952.45 & \textbf{969.57} \\
500 & 122.38 & 332.96 & 885.58 & \textbf{818.32} & \textbf{1001.43} & 89.78 & \textbf{299.07} & \textbf{881.26} & 977.05 & 1828.62 & \textbf{81.39} & 352.59 & 1208.90 & 1383.31 & 1705.60 \\
\bottomrule
\end{tabular}%
}
\end{table}

\end{document}